\documentclass[11pt]{article}
\usepackage[margin=1.3in]{geometry}
\usepackage{amsmath, amsfonts, amssymb, amsthm}
\usepackage{graphicx}
\usepackage{hyperref}
\hypersetup{
    colorlinks,%
    citecolor=black,%
    filecolor=black,%
    linkcolor=black,%
    urlcolor=black
}
\usepackage{color}

\newtheorem{thm}{Theorem}[section]
\newtheorem{prop}[thm]{Proposition}

\newtheorem{lemma}[thm]{Lemma}

\newtheorem{preremark}[thm]{Remark}

\numberwithin{equation}{section}

\newcommand{\R}{\mathbb R}
\newcommand{\eps}{\varepsilon}
\newcommand{\To}{\longrightarrow}

\newcommand{\lap} {\Delta}

\newcommand{\dd} {\; \mathrm{d}}

\newcommand{\comment}[1] {}
\DeclareMathOperator*{\osc}{osc}

\newcommand{\la}{\lambda}

\title{Propagation in a non local reaction diffusion equation with spatial and genetic trait structure\bigskip}

\author{\medskip Henri Berestycki, \quad Tianling Jin, \quad Luis Silvestre}

\date{\today}

\begin{document}

\maketitle

\begin{abstract}
We study existence and uniqueness of traveling fronts, and asymptotic speed of propagation for a non local reaction diffusion equation with spatial and genetic trait structure.
\end{abstract}

\section{Introduction}
In this article, we study  bounded non-negative solutions of reaction-diffusion equations with non-local interactions of the type:
\begin{equation}\label{eq:non local interactions}
u_t - \Delta u + \alpha g(y) u = \left(1 - \int_{\R^N} K(z) u(t,x,z) \dd z \right) u,\quad (t,x,y)\in\R^+\times\R^m\times\R^N,
\end{equation}
where $\R^+=(0,+\infty),  m\ge 1, N\ge 1$, $\Delta$ is the Laplacian operator in $(x,y)$ variables, $\alpha$ is a positive constant, $K$ and $g$ are given non-negative functions. 

This equation \eqref{eq:non local interactions} arises in population dynamic models, see, e.g., \cite{DFP}, \cite{ADP} and equations (2) and (3) in \cite{BC}. It describes a population which is structured by a set of quantitative genetic traits denoted $y \in \R^N$ and depends on the spatial location $x\in\R^m$. This population is subject to  migration, mutations, growth, selection and intraspecific competition. The term $\Delta_x u$ accounts for migration by random dispersal through Brownian motion. For simplification, we assume here that mutations also involve a random dispersion in the $y$-variables, hence the Laplacian term $\Delta_y u$ with respect to $y$ in the equation. Note that to simplify notations, we have taken the same diffusion coefficient, 1, both for the spatial diffusion and the diffusion in the trait space. The results remain unchanged if instead of $\Delta u$ in the equation above we have $ d_x \Delta_x u+d_y \Delta_y u$ with $d_x$ and $d_y$ positive constants. Next, local selection involves a fitness function represented here by the term $\alpha g(y)$. The effective growth rate is thus given by $u - \alpha g(y) u$. We assume that at every point in space the selection favors the trait $y=0$ which  translates into condition \eqref{eq:condition g} below. 
In this context, $\alpha$ can be interpreted as an intensity of genetic pressure.

Lastly, at every point $x$ in space and time $t$, each individual is subject to competition with all the individuals at the same location but with all possible values of the trait. The intensity of the competition can furthermore depend on the genetic traits of the competitors through a kernel $K= K(y)$. 
 Let $u= u(t,x,y)$ denote the density of this postulation depending on time $t$, location $x$ and trait $y$. These various effects combine into equation 
 \eqref{eq:non local interactions} for $u$.  When the intraspecific competition does not distinguish between the genetic features of competitors, the equation reads:
\begin{equation}   \label{eq:integral}
u_t - \Delta u + \alpha g(y) u = \left(1 -  k\int_{\R^N} u(t,x,z) \dd z \right) u,\quad (t,x,y)\in\R^+\times\R^m\times\R^N,
\end{equation}
 where $k>0$ is a constant.  This is  a particular case of \eqref{eq:non local interactions} above.
 
This paper is about nonnegative bounded solutions of the reaction-diffusion equation  with nonlocal interaction \eqref{eq:non local interactions}. 
We study the long time behavior of solutions and the traveling front solutions of \eqref{eq:non local interactions}.  

If not otherwise stated, we always assume the kernel $K$ to satisfy
\begin{equation}\label{eq:condition-1}
K\not\equiv 0,\ 0\le K(z)\le \kappa e^{\kappa |z|}, \quad\forall\ z\in\R^N\ \mbox{ with some fixed }\kappa>0,
\end{equation}
and $g$ to be a H\"older continuous function satisfying
\begin{equation}\label{eq:condition g}
\ g(0)=0,\ 0<g(y)\le \kappa e^{\kappa|y|} \mbox{ in } \R^N\setminus\{0\}, \mbox{ and } \lim_{|y|\to+\infty}g(y)=+\infty.
\end{equation}
With this assumption on $g$, the term $\alpha g(y) u$ expresses the preference that the most favorable trait is $y=0$. Increasing the value of $\alpha>0$ creates a tendency of the solution to decrease for all values $y \neq 0$. When $|y|$ is sufficiently large, it offsets the reproduction term in the equation in the sense that $1-\alpha g(y) <0$, and thus, the effective birth rate is negative for large $|y|$.

We will show that there is a constant $\bar \alpha$, which will be uniquely determined in Proposition \ref{lem:choice of alpha}, so that if $\alpha > \bar \alpha$, the solution $u(t,x,y) \to 0$ uniformly as $t \to +\infty$. This means that too large a genetic pressure always leads to extinction whatever the initial datum is. 

Our main results concern the case $\alpha < \bar \alpha$.  
The first one describes the planar traveling wave solutions of \eqref{eq:non local interactions}. These are solutions of the type $u(x\cdot e-ct,y)$, where $c\in\R$ is a constant, $e\in\mathbb S^{m-1}$, $u:\R\times \R^{N}\to\R$ solves 
\begin{equation} \label{e:traveling-wave-intro}  
-c u_s(s,y) - \Delta u(s,y) + \alpha g(y) u(s,y) = \left(1 - \int_{\R^N} u(s,z)K(z)  \dd z \right) u(s,y),
\end{equation}
with $s\in\R,\, y\in \R^N$ and
such that 
\begin{equation} \label{e:traveling-wave-infinity-intro} 
\lim_{s\to+\infty}u(s,\cdot)\equiv 0 \quad\mbox{and}\quad \liminf_{s\to-\infty}u(s,\cdot)>0.
\end{equation}
We also consider the stationary solution:
\begin{equation} \label{e:stat}  
 - \Delta v (x, y)  + \alpha g( y) v(x, y)  = \left(1 - \int_{\R^N} v(x,z)K(z)  \dd z \right) v(x,y), \quad (x,y) \in\R^m\times \R^N. 
\end{equation}

The results in the next theorem characterize these stationary solutions as well as the traveling wave solutions.
\begin{thm} \label{t:intro-stationary} Assume that $0<\alpha < \bar \alpha$. 
There exists a positive number $c^*$ such that
\begin{itemize} 
\item
There exists a unique positive bounded stationary solution $v(x,y)$ of \eqref{eq:non local interactions}, that is a solution of \eqref{e:stat}. Moreover, this stationary solution is independent of $x$. We denote it $v= V(y)$.
\item If $0\leq c < c^*$, there exists a unique positive bounded solution of \eqref{e:traveling-wave-intro}. Moreover, the solution is independent of $s$. Therefore, it is equal to $V(y)$. 
\item For all $c \geq c^*$, there exists a unique nonnegative bounded solution $u$ of \eqref{e:traveling-wave-intro} such that \eqref{e:traveling-wave-infinity-intro} holds. Moreover, in this case,  $\lim_{s\to-\infty}u(s,y)=V(y)$ uniformly in $y$.
\end{itemize}
\end{thm}
We will see in the proof that $c^*$ is explicitly given in terms of the ground state energy of the operator $-\Delta + \alpha g(y)$. 
Theorem \ref{t:intro-stationary} provides Liouville type results. The first part asserts the uniqueness of the stationary solution. The second one states that there are no non trivial traveling wave solutions of speed less than $c^*$. The third part says that there is a unique traveling wave for all speeds faster than or equal to $c^*$. The proof of Theorem \ref{t:intro-stationary} is obtained by a direct combination of Theorem \ref{thm:existence and uniqueness}, Theorem \ref{thm:liouville with c} and Theorem \ref{Liouville-stat} in the main body of the paper. 

Note that when $\alpha > \bar \alpha$ we can still make sense of the previous result by considering $c^* = +\infty$. If $\alpha < \bar \alpha$, then $c^*$ is a finite positive number.

The second main result concerns the asymptotic speed of propagation of general non negative solutions of \eqref{eq:non local interactions}. Recall that $V(y)$ is the unique positive bounded stationary solution of \eqref{eq:non local interactions} in Theorem \ref{t:intro-stationary}.

\begin{thm} \label{t:into-asp}
We assume $0<\alpha< \bar \alpha$.
Let $u$ be a solution of \eqref{eq:non local interactions} so that $u(0,x,y) = u_0(x,y)$ is smooth, nonnegative, compactly supported and $u_0\not\equiv 0$. Assume \eqref{eq:condition-1}, \eqref{eq:condition g} and also that $K$ is bounded below in a neighborhood of the origin (condition \eqref{eq:assumption-2} below). Then, there is invasion by $V(y)$ which means that at every point $(x,y)\in \R^m\times \R^N$ one has $\lim_{t\to \infty} u(t,x,y) = V(y)$. Furthermore, 
 the asymptotic speed of propagation is equal to $c^*$ in the sense that
\[
\lim_{t\to+\infty}\Big(\sup_{|x|\ge ct,\ y\in\R^N}u(t,x,y)\Big)=0\quad\mbox{ for all }c>c^*.
\]
and
\[  
\lim_{t\to+\infty}\Big(\sup_{|x|\le ct,\ y\in\R^N}|u(t,x,y)-V(y)|\Big)=0\quad\mbox{ for all }0\le c<c^*.
\]
\end{thm}


There is a large literature devoted to local reaction-diffusion equations. When the competition term is replaced by a local one, \eqref{eq:non local interactions}  reduces to  
\begin{equation}\label{eq:BC}
u_t - \Delta u + \alpha g(y) u = f(u).
\end{equation}
In \cite{BC}, H. Berestycki and G. Chapuisat study this  local equation. They establish the existence and characterization of traveling fronts, asymptotic speed of propagation and other related properties when $f$ is a nonlinearity of either Fisher-KPP type or bistable type.
The methods of \cite{BC} rely essentially on the maximum principle and comparison principles for parabolic equations. Therefore, they fall short for non-local equations as the one of interest here. 

 Several works address the questions of existence of traveling wave solutions and  asymptotic speed of propagation for reaction-diffusion equations with nonlocal reaction terms related to \eqref{eq:non local interactions}. Using a topological degree argument and a priori estimates, M. Alfaro, J. Coville and G. Raoul \cite{ACR} prove the existence of traveling waves for the equation \eqref{eq:non local interactions} with $K$ more general than here in that it also depends on $y$, i.e. $K=K(y,z)$, but with further restrictions on the growth of $K$. In particular, they assume that $k_1 \leq K(y,z) \leq k_2$ for all $y,z$, where $k_1, k_2$ are two positive constants. E. Bouin and V. Calvez \cite{BCalvez} (see also \cite{BCMMPRV}), constructed traveling wave solutions of equations for bounded traits with Neumann boundary conditions, where the space-diffusivity depends on the trait and the competition kernel is $K\equiv 1$. However, both of these works do not prove uniqueness results for the traveling waves and the asymptotic profile is not specified. E. Bouin and S. Mirrahimi \cite{BM} derive certain asymptotic speeds of propagation, and asymptotic behavior of either $u$ or the average of $u$ in the trait $y$, 
for equations with bounded traits and Neumann boundary conditions by a Hamilton-Jacobi approach. Very recently, we learned from O. Turanova \cite{OT} that she has generalized these results to equations of trait dependent space-diffusivity as those in \cite{BCalvez}. Lastly, N. Berestycki, C. Mouhot and G. Raoul \cite{BMR} establish a propagation law in $t^{3/2}$ for the model of \cite{BCalvez} for toads invasion. The paper \cite{BCMMPRV} provides a heuristic analysis and numerical computations for this model.

In Section \ref{sec:othermodels} we consider some variations of the model \eqref{eq:non local interactions} and extend our existence and uniqueness results. 
We analyze in particular the case in which the trait space is bounded, and also the case in which the diffusion in trait is fractional.

Another related nonlocal Fisher-KPP equation arises in ecology with a convolution term. This equation is of the form
\begin{equation}\label{eq:convolution}
 u_t-\Delta u=u(1-\phi*u)
\end{equation}
where the nonlocal competition is given by a convolution with a kernel $\phi$. The papers \cite{BNPR, Britton, G, HR, WLR} and other works mentioned therein study the steady states, traveling waves and asymptotic speeds of propagation for \eqref{eq:convolution}. \\

The paper is organized as follows. In Section \ref{sec:tr} we prove Theorem \ref{t:intro-stationary} on the existence and uniqueness of traveling fronts of \eqref{eq:non local interactions}. In Section \ref{sec:asp} we show Theorem \ref{t:into-asp} on the asymptotic speed of propagation. 

An important ingredient in Section \ref{sec:asp} is a uniform pointwise bound for the solutions. One main difference between \eqref{eq:BC} and \eqref{eq:non local interactions} (as well as \eqref{eq:convolution}) is that in general we do not have comparison principles for solutions of  \eqref{eq:non local interactions} (nor \eqref{e:traveling-wave-intro}). Thus, many arguments used for the classical Fisher-KPP equation or for \eqref{eq:BC} as in \cite{BC} in general do not apply. 

\bigskip

\noindent\textbf{Acknowledgements:} The research of H. Berestycki leading to these results has received funding from the European Research Council under the European Union's Seventh Framework Programme (FP/2007-2013) / ERC Grant Agreement n.321186 - Reaction-Diffusion Equations, Propagation and Modelling, and from the French National Research Agency (ANR), within the project NONLOCAL ANR-14-CE25-0013. Part of this work was done while H. Berestycki was visiting the Department of Mathematics at the University of Chicago and was also supported by NSF grant DMS-1065979. T. Jin was supported in part by NSF grant DMS-1362525, and he would like to thank Professor YanYan Li for his interests and constant encouragement. L. Silvestre was supported in part by NSF grants DMS-1254332 and DMS-1065979.

\section{Existence and uniqueness of traveling fronts}\label{sec:tr}

In this section, we will study existence and uniqueness of planar traveling fronts of \eqref{eq:non local interactions}, which are solutions of \eqref{e:traveling-wave-intro}. This is actually equivalent to the case when the spatial dimension $m=1$. Let us abuse the notations a little: we replace the variable $s$ by $x$ in the equation \eqref{e:traveling-wave-intro}. Therefore, in this section, $x\in\R$ (not $\R^m$), and we study solutions $u(x,y):\R\times \R^{N}\to\R$ of
\begin{equation} \label{e:traveling-wave}  
-c u_x(x,y) - \Delta u(x,y) + \alpha g(y) u(x,y) = \left(1 - \int_{\R^N} u(x,z)K(z)  \dd z \right) u(x,y)
\end{equation}
such that 
\begin{equation} \label{e:traveling-wave-infinity} 
\lim_{x\to+\infty}u(x,\cdot)\equiv 0 \quad\mbox{and}\quad \liminf_{x\to-\infty}u(x,\cdot)>0.
\end{equation}
One important observation is that when the solution $u$ of \eqref{e:traveling-wave} has the special structure $u(x,y) = v(x) \psi(y)$, where $\psi$ is an eigenfunction of the left hand side of the equation, then the function $v$ satisfies a classical KPP-Fisher reaction diffusion equation in $x$. The main difficulty in this section is to show that all traveling wave solutions $u$ must have this separated variables structure.

\subsection{A spectral lemma and asymptotic profiles}
To start with, for $g$ satisfying \eqref{eq:condition g}, we define the Hilbert space
\[
\mathcal H(\R^N)= \{v\in H^1(\R^N): \sqrt{g} v\in L^2(\R^N)\},
\]
with its associated inner product
\[
\langle u,v\rangle=\int_{\R^N} \nabla u \nabla v+g u v\dd y.
\]
We denote its norm as
\[
\|v\|_{\mathcal H(\R^N)}=\left(\int_{\R^N} |\nabla v|^2+g v^2\dd y\right)^{\frac 12}.
\]
Since $g$ is bounded from below by a positive constant in the complement 
$B^c$ of the unit ball in $\R^N$, it is easily seen that 
$\mathcal H(\R^N) \hookrightarrow  H^1(\R^N)$ with a continuous injection. 

The following lemma is elementary. We include its proof here for completeness.

\begin{lemma}
The embedding
$
\mathcal H(\R^N)\hookrightarrow L^2(\R^N)$ is compact.
\end{lemma}

\begin{proof} Let $\{v_n\}$ be a bounded sequence in $\mathcal H(\R^N)$. By the assumption \eqref{eq:condition g}, $\forall~\eps>0$, there exists $R_{\eps}>0$ such that
\[
\|v_n\|_{L^2(B^c_{R_\eps})}<\eps\quad\mbox{for all }n.
\]
Using the Rellich-Kondrachov theorem, there exists a subsequence $\{v_{k_n}\}$ such that
\[
\limsup_{n,m\to\infty}\|v_{k_n}-v_{k_m}\|_{L^2(B_{R_\eps})}=0.
\]
It follows that
\[
\limsup_{n,m\to\infty}\|v_{k_n}-v_{k_m}\|_{L^2(\R^N)}<\eps.
\]
Finally, we can use a standard diagonal argument to extract a subsequence $\{v_{k_n}\}$ satisfying
\[
\limsup_{n,m\to\infty}\|v_{k_n}-v_{k_m}\|_{L^2(\R^N)}=0.
\]
This finishes the proof.
\end{proof}

Let $\mathcal L$ be the linear operator:
\[ \mathcal Lf := -\Delta_y f + \alpha g f,\]
where $\Delta_y$ denotes the Laplacian operator in the variable $y$ only.
\begin{lemma} \label{l:spectral}
The spectrum of $\mathcal L$ consists only of eigenvalues. All its eigenvalues are positive and we can write them in a monotone increasing sequence $\{0 < \lambda_0 < \lambda_1 \leq \lambda_2 \leq \dots\}$ so that $\lambda_i \to \infty$ as $i \to \infty$. The first eigenvalue $\lambda_0$ is simple and corresponds to a positive eigenfunction $\psi_0$. All the other eigenfunctions $\psi_i$ change signs. The eigenfunctions $\{\psi_i\}$ form an orthonormal  basis of $L^2(\R^N)$.
\end{lemma}
\begin{proof}
It follows from Riesz representation theorem that for each $f\in L^2(\R^N)$, there exists a unique function $u\in\mathcal H(\R^N)$ solving
\begin{equation} \label{e:L} 
\mathcal Lu=f
\end{equation}
in the sense that
\[
\int_{\R^N}\nabla u\nabla v+guv=\int_{\R^N}f v\quad\mbox{for all }v\in\mathcal H(\R^N).
\]
We write
\[
u=\mathcal L^{-1} f.
\]
Then $\mathcal L^{-1}$ is the operator which maps the right hand side $f$ to the solution $u$ in \eqref{e:L}. This operator is naturally bounded from $L^2(\R^N)$ to $\mathcal{H}(\R^N)$.

Since the embedding
\[
\mathcal H(\R^N)\hookrightarrow L^2(\R^N)\quad\mbox{is compact},
\]
we have that
\[
\mathcal L^{-1}: L^2(\R^N)\to  L^2(\R^N)\quad\mbox{is compact}.
\]
Then the conclusion follows immediately from the standard spectral theorem for compact symmetric operators on Hilbert spaces. Since the eigenfunctions are mutually orthogonal and $\psi_0>0$, all the other eigenfunctions will change signs.
\end{proof}

Since the first eigenvalue of $\mathcal L$ has monotonic and continuous dependence on $\alpha$, we have
\begin{lemma}\label{lem:choice of alpha}
There exists some $\bar\alpha$ such that $\lambda_0(\alpha)<1$ when $\alpha<\bar\alpha$, $\lambda_0(\bar\alpha)=1$ and $\lambda_0(\alpha)>1$ when $\alpha>\bar\alpha$.
\end{lemma}
\begin{proof}
See Proposition 1 and Corollary 2 in \cite{BC} for the detailed proof.
\end{proof}

Also, we have the following estimates for the first eigenfunction.
 \begin{prop}\label{prop:eigenfunction exp in y}
For every $\gamma>0$ there exists $C>0$ such that
\[
0\le \psi_0(y)\le C e^{-\gamma |y|}\quad\mbox{and}\quad |\nabla \psi_0(y)|\le C e^{-\gamma |y|}.
\]
\end{prop}
\begin{proof}
The function $\psi_0$ satisfies the equation
\[
\mathcal L\psi_0=\lambda_0 \psi_0.
\]
Then the bound for $\psi_0$ follows from Lemma 2.2 in  \cite{BR}. Consequently, it follows from the gradient estimates for the Laplace operator that
\[
|\nabla \psi_0(y)|\le 2\sup_{|z-y|=1}|\psi_0(z)|+\sup_{|z-y|\le 1}(\alpha g(z)+\lambda_0)|\psi_0(z)|.
\]
Therefore, the bound for $|\nabla \psi_0|$ follows from the bound for $\psi_0$.
\end{proof}

Moreover,  we have
\begin{prop}\label{prop:exp in y}
Let $u$ be a nonnegative bounded solution of \eqref{e:traveling-wave} such that $u\not\equiv 0$. Then for every $\gamma>0$ there exists $C>0$ such that for all $(x,y)\in \R\times\R^N$,
\[
0< u(x,y)\le C e^{-\gamma |y|}\quad\mbox{and}\quad |\nabla u(x,y)|\le C e^{-\gamma |y|}.
\]
\end{prop}
\begin{proof}
Let $\varphi_0(y)=e^{-\gamma |y|}+\eps e^{\gamma |y|}$ for some $\eps\in (0,1)$. Then for $|y|\ge r_0:=(N-1)\gamma^{-1}$, we have
\[
 \Delta_y \varphi_0(y)=(\gamma^2-(N-1)\gamma|y|^{-1}) e^{-\gamma |y|}+\eps(\gamma^2+(N-1)\gamma|y|^{-1}) e^{\gamma |y|}\le2\gamma^2\varphi_0(y).
\]
Let $w(x,y)=C(1+\eps x \arctan x)(e^{-\gamma |y|}+\eps e^{\gamma |y|})$ for some $C>0$. Then
\begin{equation}\label{eq:aux func exp}
\begin{split}
&-cw_x-\Delta w+\alpha g w-w\\
&=C\Big(\eps\big(-c(\arctan x+\frac{x}{1+x^2})-\frac{2}{(1+x^2)^2}\big)\\
&\quad\quad +(\alpha g(y)-\varphi_0(y)^{-1}\Delta_y\varphi_0(y)-1)(1+\eps x \arctan x)\Big)(e^{-\gamma |y|}+\eps e^{\gamma |y|})\\
&\ge C\left(-3|c|-2+\alpha g(y)-2\gamma^2-1\right)(e^{-\gamma |y|}+\eps e^{\gamma |y|})\\
&\ge 0
\end{split}
\end{equation}
for $(x,y)\in \Omega:=\R\times \{y: |y|\ge R_0\}$, where $R_0\ge r_0$ is chosen such that $g(y)\ge\alpha^{-1}(2\gamma^2+3+3|c|)$ for all $|y|\ge R_0$.  It follows from \eqref{e:traveling-wave} that
\[
-c u_x - \Delta u + \alpha g u \le u\quad\mbox{in }\Omega.
\]
Since $u$ is a bounded function, we can choose $C$ large so that  $Ce^{-\gamma R_0}\ge \sup_{x,y} u(x,y)$. Meanwhile, $\alpha g>3$ in $\Omega$. From the maximum principle we infer that
\[
u\le w\quad\mbox{in }\Omega.
\]
By sending $\eps\to 0$, we have
\[
0\le u(x,y)\le C e^{-\gamma |y|}\quad\mbox{in }\R\times\R^N.
\]
Consequently, by the gradient estimates we have
\[
|\nabla u(x,y)|\le C e^{-\gamma |y|}\quad\mbox{in }\R\times\R^N.
\]
Meanwhile, it follows from strong maximum principle that $u>0$ in $\R\times\R^N$.
\end{proof}

Let us consider the steady states of \eqref{e:traveling-wave}, i.e., the nonnegative bounded solutions of
\begin{equation}\label{eq:steady}
  -\Delta_y V(y)+\alpha g(y)V(y)= \left(1-\int_{\R^N} V(z)K(z)\dd z\right)V(y).
\end{equation}
 From Proposition \ref{prop:exp in y} and Lemma \ref{l:spectral} we see that $1-\int_{\R^N} V(y)K(y)\dd y$ is the first eigenvalue of $\mathcal L$ and $V\in\mathcal H(\R^N)$ is an eigenfunction.   Thus, when $\alpha\ge\bar\alpha$ for $\bar\alpha$ be the one in Lemma \ref{lem:choice of alpha}, every nonnegative bounded solution of \eqref{eq:steady} has to be identically zero. When $\alpha\in (0,\bar\alpha)$ and $V\not\equiv 0$, then
\begin{equation}\label{eq:steady solution}
 V=\mu\psi_0,
\end{equation}
where
\begin{equation}\label{eq:-infty}
\mu=(1-\lambda_0) \left(\int_{\R^N} \psi_0(y)K(y)\dd y\right)^{-1},
\end{equation}
$\lambda_0\in (0,1) $ is the first eigenvalue and $\psi_0$ is the first eigenfunction in Lemma \ref{l:spectral}. 

In general it is not clear whether the asymptotic profiles of traveling wave solutions to nonlocal equations are solutions of the steady equations like \eqref{eq:steady}. See, e.g., \cite{BNPR}. However, we will show in Theorem \ref{thm:existence and uniqueness} that it is the case for \eqref{e:traveling-wave}. That is, the solution of \eqref{e:traveling-wave} and \eqref{e:traveling-wave-infinity} will satisfy
\[
\lim_{x\to-\infty} u(x,\cdot)=V.
\]

\subsection{Reduction to the classical Fisher-KPP equation}\label{sec:reduction}

Let $u$ be a nonnegative bounded solution of \eqref{e:traveling-wave} such that  $u\not\equiv 0$. Then $u>0$ everywhere. Let $b : \R \to \R$ be its integral in $y$:
\[ b(x) = \int_{\R^N} u(x,z) K(z)\dd z.\]
Since $u>0$ and $K\not\equiv 0$, we have $b(x)>0$ for all $x\in\R$. 
For the rest of this section it is convenient to forget the relationship between $u$ and $b$. We will only take into consideration that $u$ solves the linear equation
\begin{equation} \label{e:linearized-traveling-wave}  -c u_x - \Delta u + \alpha g u = (1-b(x)) u
\end{equation}
where $b$ is some positive bounded function.
Because of Lemma \ref{l:spectral} and Proposition \ref{prop:exp in y}, we can write $u$ as
\begin{equation}\label{e:decomposition}
u(x,y) = \sum_{i=0}^\infty v_i(x) \psi_i(y),
\end{equation}
where \[v_i(x)=\int_\R u(x,z)\psi_i(z)\dd z\in L^\infty(\R).\]  
By Proposition \ref{prop:exp in y} again, it is easy to verify that the equation \eqref{e:linearized-traveling-wave} splits into a sequence of equations for each $v_i$:
\begin{equation} \label{e:equations-for-vi}  
-c \partial_x v_i - \partial_{xx} v_i + \lambda_i v_i = (1-b(x)) v_i.
\end{equation}

\begin{lemma}\label{l:finite summation}
If $\lambda_i > 1$, then every bounded solution of \eqref{e:equations-for-vi} has to be identically zero.
\end{lemma}

\begin{proof}
Let
\[
 w(x)=e^{\frac{-c-\sqrt{c^2+4(\lambda_i-1)}}{2}x}+e^{\frac{-c+\sqrt{c^2+4(\lambda_i-1)}}{2}x}
\]
Since $\lambda_i-1 > 0$, then $w(x)\to+\infty$ as $x\to\pm\infty$. Moreover
\[
-c \partial_x w - \partial_{xx} w + (\lambda_i-1) w=0.
\]
Thus, for every $\eps>0$, we have
\[
 -c \partial_x (\eps w-v_i) - \partial_{xx} (\eps w-v_i) + (\lambda_i+b(x)-1) (\eps w-v_i)=b\eps w .
\]
Since $v_i$ is bounded and $b>0$ in $\R$,  the maximum principle yields that $v_i\le\eps w$. Similarly, $-v_i\le\eps w$. By sending $\eps\to 0$, we have $v_i\equiv 0$.
\end{proof}

The previous lemma tells us that there can be only finitely many terms in the expression for $u$: 
\begin{equation} \label{e:sum-for-u}  
u(x,y) = \sum_{i=0}^J v_i(x) \psi_i(y),
\end{equation}
where $J$ is a positive integer. Moreover, for each $i=1,\dots,J$, we have $\lambda_i\le 1$.  We suppose that $ v_i\not\equiv 0$ for all $i=0,1,\dots,J$. Note that $v_0(x)>0$ for all $x\in\R$.

\begin{lemma}\label{l:quotient is bounded}
For each $i=1,\cdots,J$, we have
\[
 w_i:=\frac{v_i}{v_0}\in L^\infty(\R).
\]
\end{lemma}
\begin{proof}
 Suppose that for some $k\in \{1,\dots,J\}$, $w_k$ is not bounded in $\R$. Then there exists a sequence $\{x_j\}$, $|x_j|\to\infty$, such that $|w_k(x_j)|\to\infty$. Then there exist $\ell\in\{1,\dots,J\}$, and a subsequence of $\{x_j\}$ which will be still denoted as $\{x_j\}$, such that
\[
 |v_l(x_j)|=\max_{1\le i\le J}|v_i(x_j)|\quad\mbox{ for all }j,\quad\mbox{and thus }|w_\ell(x_j)|\to\infty.
\]
We may also assume that all $\{v_l(x_j)\}$ have the same sign. From \eqref{e:sum-for-u} we get
\[
 \frac{u(x_j,y)}{v_\ell(x_j)}=\frac{1}{w_\ell(x_j)}\psi_0(y)+\sum_{i=1,i\neq \ell}^J \frac{v_i(x_j)}{v_\ell(x_j)} \psi_i(y)+\psi_\ell(y).
\]
Subject to taking a subsequence of $\{x_j\}$, we have
\[
 \lim_{j\to\infty} \frac{u(x_j,y)}{v_\ell(x_j)}=\sum_{i=1, i\neq \ell}^J \tau_i \psi_i(y)+\psi_\ell(y),
\]
where each $\tau_i\in\R$. This is a contradiction since the right-hand side changes signs (since it is orthogonal to $\psi_0$) but the left-hand side does not change sign.
\end{proof}

From now on, let $c\ge c^*=2\sqrt{1-\lambda_0}$. Then  $c>2\sqrt{1-\lambda_i}$ for all $i=1,\dots,J$. In this case, we will have a lower bound of $|v_i|$ near $+\infty$.

\begin{lemma} \label{l:term-exp-decay-low}
There exists some $x_0>0$ such that $v_i(x)$ does not change signs in $[x_0,+\infty)$. Moreover,
\[ \liminf_{x\to+\infty}|v_i(x)| e^{\gamma_i x}>0.\]
where 
\[ \gamma_i = \frac{c - \sqrt{c^2 - 4(1-\lambda_i)}} 2 \ge 0.\]
\end{lemma}
\begin{proof}
We argue by contradiction. Suppose there exist sequences $x_n\to+\infty$, $\eps_n\to 0$ with
\[
 |v_i(x_n)|<\eps_n e^{-\gamma_i x_n}.
\]
Without loss of generality, we can assume $v_i(x_n)\ge 0$, since otherwise we can consider $-v_i$ instead. Let $\tilde \gamma_i=\frac{c + \sqrt{c^2 - 4(1-\lambda_i)}} 2>\gamma_i\ge 0$. Note that in the following argument, $\tilde \gamma_i$ is not needed unless $\gamma_i=0$.

We claim the following
\begin{equation} \label{eq:less}
  v_i(x)<\eps_n (e^{-\gamma_i x}+e^{-\tilde \gamma_i x})\quad\forall\ x\in (-\infty,x_n].
\end{equation}
Indeed, suppose that there exists some $x'\in (-\infty,x_n)$ such that   $v_i(x')\ge \eps_n (e^{-\gamma_i x'}+e^{-\tilde \gamma_i x'})$. Since $v_i$ is bounded and $w(x)\to+\infty$ as $x\to-\infty$, there exists some constant $C\ge 1$ such that $w(x)=C\eps_n (e^{-\gamma_i x}+e^{-\tilde \gamma_i x})$ touches $v_i$ from above in $(-\infty,x_n]$ at some point $\bar x\in (-\infty,x_n)$. Since
\[ -c w_x - w_{xx} = (1-\lambda_i) w,\]
we have
\[
 -c (w-v_i)_x - (w-v_i)_{xx} +(\lambda_i+b-1)(w-v_i) = b w.
\]
But this is impossible if we evaluate the above equation at $\bar x$ since $b(x)>0$ in $\R$. 

Now we can let $n\to\infty$ in \eqref{eq:less} to obtain $v_i\le 0$ in $\R$. By applying the same arguments to $-v_i$, we obtain $v_i\equiv 0$, which is a contradiction. 
\end{proof}
Under the extra assumption $b(x)\to 0$ as $x\to+\infty$, we will have an upper bound of $|v_i|$ near $+\infty$.
\begin{lemma} \label{l:term-exp-decay-upper}
Suppose $b(x)\to 0$ as $x\to+\infty$. For all $\delta>0$ we have
\[ \limsup_{x\to+\infty}|v_i(x)| e^{(\gamma_i-\delta) x}<+\infty.\]
\end{lemma}

\begin{proof}
 Suppose $v_i(x)>0$ and $0\le b(x)\le\delta^2$ for $x\in[x_0,+\infty)$. Let $\tilde v_i=e^{cx/2}v_i$. Then
 \[
 \partial_{xx}\tilde v_i=(\lambda_i+b-1+c^2/4)\tilde v_i.
 \]
 Let $w_i$ be the solution of
 \[
 \begin{cases}
  \partial_{xx}w_i=(\lambda_i+\delta^2-1+c^2/4)w_i,\\
  w_i(x_0)=\tilde v_i(x_0),\   \partial_x w_i(x_0)=\partial_x \tilde v_i(x_0)+1.
  \end{cases}
 \]
 Then $w_i>\tilde v_i$ near $x_0$. We claim that $w_i>\tilde v_i$ for all $x\in[x_0,+\infty)$. If not, let $x_1\in[x_0,+\infty)$ be the smallest value such that $w_i(x_1)=\tilde v_i(x_1)$. Thus, $\partial_x w_i(x_1)\le \partial_x \tilde v_i(x_1)$. Then we have
 \[
 \int_{x_0}^{x_1} (b-\delta^2)w_i\tilde v_i =\int_{x_0}^{x_1}  w_i\partial_{xx}\tilde v_i- \tilde v_i \partial_{xx}w_i = (w_i\partial_x \tilde v_i- \tilde v_i \partial_x w_i)\lvert_{x_0}^{x_1}>0.
 \]
 This is a contradiction since $b(x)\le\delta^2$ for $x\in[x_0,+\infty)$. Hence
 \[
 \tilde v_i\le w_i\le C e^{\frac{\sqrt{c^2+4\lambda_i-4+4\delta^2}}{2}x}\le C e^{(\frac{\sqrt{c^2+4\lambda_i-4}}{2}+\delta)x}\le C e^{(\frac c2 -\gamma_i+\delta)x}.
 \]
By the definition of $ \tilde v_i$, we have
 \[
 v_i(x) \le C e^{-(\gamma_i-\delta) x}.
 \]
 This finishes the proof.
\end{proof}

By combining the above three lemmas, we will conclude that $J=0$ in the expansion \eqref{e:sum-for-u} if $u(x,\cdot )\to 0$ as $x\to+\infty$.

\begin{lemma} \label{l:one-term}
Let $u$ be a nonnegative bounded solution of \eqref{e:traveling-wave} with  $u\not\equiv 0$. Suppose in addition that for each $y\in\R^N$, $u(x,y)\to 0$ as $x\to+\infty$. Then the only non zero term in \eqref{e:sum-for-u} is the one with $i = 0$.
\end{lemma}

\begin{proof}
By Proposition \ref{prop:exp in y} and dominated convergence theorem, we have
\[
 b(x)=\int_{\R^N} u(x,y)K(y)\dd y\to 0\quad\mbox{as }x\to+\infty.
\]
Therefore, by Lemma \ref{l:term-exp-decay-low} and Lemma \ref{l:term-exp-decay-upper} we have for every $i=1,\dots,J$,
\[
\lim_{x\to+\infty}\frac{|v_i(x)|}{|v_0(x)|}=+\infty,
\]
since $\gamma_i < \gamma_0$ if $\lambda_0 < \lambda_i$.  This is in contradiction with Lemma \ref{l:quotient is bounded}.
\end{proof}

\begin{thm}\label{thm:existence and uniqueness}
Let $\alpha\in (0,\bar\alpha).$ If $c\ge c^*=2\sqrt{1-\lambda_0}$ then there exists a unique nonnegative bounded solution of \eqref{e:traveling-wave} satisfying \eqref{e:traveling-wave-infinity}. Moreover, $\lim_{x\to-\infty}u(x,y)=V(y)$, and the convergence at both $-\infty$ and $+\infty$ is uniform in $y$.
\end{thm}

\begin{proof}
After Lemma \ref{l:one-term}, we reduce the problem to functions $u$ of the form
\[ u(x,y) = v_0(x) \psi_0(y).\]
From \eqref{eq:-infty} we see that $\int_{\R^N} \psi_0(y) K(y)\dd y=(1-\lambda_0)\mu^{-1}$, and thus  $v=\mu^{-1}v_0$ satisfies
\begin{equation}\label{eq:reduce to KPP}
\begin{cases}
-c \partial_x v - \partial_{xx} v = (1-\lambda_0) v (1-v),\\
\liminf_{x\to-\infty}v(x)>0,\quad \lim_{x\to+\infty}v(x)=0.
\end{cases}
\end{equation}
Now once we show $\lim_{x\to-\infty}v(x)=1$, Theorem \ref{thm:existence and uniqueness} will follow from the results on the existence and uniqueness for solutions of the classical Fisher-KPP model \cite{KPP}. 

Let
\[
m=\liminf_{x\to-\infty}v(x)>0,\ M=\limsup_{x\to-\infty}v(x)<+\infty.
\]
For $x_1<x_2<0$, we integrate the first equation in \eqref{eq:reduce to KPP} to obtain 
\[
c(v(x_1)-v(x_2))+\partial_x v(x_1)-\partial_x v(x_2)=\int_{x_1}^{x_2} -c \partial_x v - \partial_{xx} v = \int_{x_1}^{x_2}(1-\lambda_0) v (1-v).
\]
It follows from Proposition \ref{prop:exp in y} that both $v$ and $\partial_x v$ are bounded functions. Thus, the left-hand side of the above equation is bounded. Therefore, we must have $m\le 1\le M$.

Therefore, we only need to show $m=M$. Suppose to the contrary that $m<M$. There exist two sequences $x_n\to-\infty$ and $z_n\to-\infty$ satisfying $z_{n+1}<x_{n+1}<z_n<x_n$ for all $n$, such that
\[
m=\lim_{n\to\infty} v(x_n),\quad M=\lim_{n\to\infty} v(z_n).
\]
Since $m<M$, there exist another two sequences $\{\tilde x_n\}$ and $\{\tilde z_n\}$ satisfying $x_{n+1}<\tilde z_n<x_n$, $z_{n+1}<\tilde x_{n+1}<z_n$, such that each $\tilde z_n$ is the maximum point of $v$ in $(x_{n+1},x_n)$ and each $\tilde x_n$ is the minimum point of $v$ in $(z_{n+1},z_n)$. Consequently,
\[
m=\lim_{n\to\infty} v(\tilde x_n),\quad M=\lim_{n\to\infty} v(\tilde z_n).
\]
By evaluating the first equation in \eqref{eq:reduce to KPP} at $\tilde x_n$ and $\tilde z_n$, and sending $n\to\infty$, we obtain
\[
M(1-M)\ge 0\quad\mbox{and}\quad m(1-m)\le 0.
\]
This contradicts $m<M$. 
\end{proof}

\subsection{Non-existence of traveling fronts}
In this subsection, we are going to show that when $\alpha\in(0,\bar\alpha)$, i.e., $\lambda_0<1$,  every bounded positive solution of \eqref{e:traveling-wave} for $c< c^*$ has to be the steady solution $V$ in \eqref{eq:steady solution}.

\begin{lemma}\label{lem:v0 lower bound}
Let $c\in [0,c^*)$. Then
\[
\inf_{x\in\R}v_0(x)>0.
\]
\end{lemma}
\begin{proof}
We argue by contradiction. Suppose that there exists a sequence $\{x_k\}$, $|x_k|\to\infty$, along which
\[
v_0(x_k)\to 0.
\]
Since $c<c^*=2\sqrt{1-\lambda_0}$, we can choose $\delta>0$ so as to have $c<2\sqrt{1-\lambda_0-\delta}$. Let $-\frac{c}{2}+i \frac{\pi}{2L}$  be the complex root of $X^2+cX +1-\lambda_0-\delta=0$ with $L>0$. We first claim that
\begin{equation}\label{eq:aux for v0}
\lim_{k\to\infty}\sup_{|x-x_k|\le L} v_0(x)=0.
\end{equation}
Indeed, we consider the translations of $v_0$ and $b$:
\[
v_0^{(k)}(x)=v_0(x+x_k),\quad b^{(k)}(x)=b(x+x_k).
\]
It follows from Proposition \ref{prop:exp in y} that all $v_0, v_0', b$ and $b'$ are bounded. After extraction of a subsequence, $v_0^{(k)}$ and $b^{(k)}$ converge to $\bar v$ and $\bar b$, respectively, locally uniformly. The limits satisfy
\[
 -c \partial_x \bar v- \partial_{xx} \bar v+ (\lambda_0+\bar b(x)-1) \bar v=0.
\]
Moreover, $\bar v\ge 0$ in $\R$ and $\bar v(0)=0$. The strong maximum principle then shows $\bar v\equiv 0$. Thus $v_0^{(k)}$ converges uniformly on $[-L,L]$ to $0$, which finishes the proof of the above claim.

Let
\[
\phi_k(x)=e^{-c (x-x_k)/2}\cos\left(\frac{\pi}{2L}(x-x_k)\right),
\]
which satisfies
\[
 -c \partial_x \phi_k- \partial_{xx} \phi_k+ (\lambda_0+\delta-1) \phi_k=0.
\]
There exists $\eps>0$ such that $\eps\phi_k$ touches $v_0$ from below at some point $\bar x_k\in (x_k-L,x_k+L)$. Then by evaluating the equation
\[
  -c \partial_x (v_0-\eps\phi_k)- \partial_{xx} (v_0-\eps\phi_k)+ (\lambda_0-1) (v_0-\eps\phi_k)=-b v_0+\delta\eps\phi_k\quad\mbox{at }\bar x_k,
\]
we have $b(\bar x_k)\ge\delta$. By Proposition \ref{prop:exp in y} we get the existence of $R>0$ (independent of $k$) such that
\[
\int_{|y|> R} u(\bar x_k,y)K(y)\dd y\le\frac{\delta}{2}\quad\mbox{for all } k.
\]
Thus by \eqref{eq:condition-1},
\[
\kappa e^{\kappa R}\int_{|y|\le  R} u(\bar x_k,y)\dd y\ge \int_{|y|\le  R} u(\bar x_k,y)K(y)\dd y\ge \frac{\delta}{2}\quad\mbox{for all } k.
\]
It follows that
\[
\sup_{|x-x_k|\le L} v_0(x)\ge v_0(\bar x_k)\ge \int_{|y|\le  R} u(\bar x_k,y)\psi_0(y)\dd y\ge \min_{|y|\le R}\psi_0(y) \frac{\delta e^{-\kappa R}}{2\kappa}\quad\mbox{for all } k.
\]
This contradicts \eqref{eq:aux for v0}.
\end{proof}

\begin{thm}\label{thm:liouville with c}
 Let $\alpha\in (0,\bar\alpha)$, $c\in [0,c^*)$ and $u$ be a nonnegative bounded solution of \eqref{e:traveling-wave} with $u\not\equiv 0$. Then $u\equiv V$.
\end{thm}
\begin{proof}
If $u\not\equiv 0$, then $u>0$. We decompose $u$ as in \eqref{e:decomposition}. By Lemma \ref{l:finite summation}, we have \eqref{e:sum-for-u} holds for some $J$. Moreover, from Lemma \ref{l:quotient is bounded} we know that $w_i=v_i/v_0$ is a bounded function for every $i=1,\dots,J$. 

 From \eqref{e:equations-for-vi} we get
 \begin{equation}\label{eq:wi}
 \partial_{xx}w_i+c \partial_x w_i+\frac{2\partial_x v_0}{v_0}\partial_x w_i=(\lambda_i-\lambda_0)w_i.
 \end{equation}
This implies that $w_i$ cannot have a positive local maximum, and $w_i$ cannot have a negative local minimum. This reduces to the following structures of $w_i$. The function $w_i$ must be either monotone (increasing or decreasing) or have only one local extrema. In the later case it would be either one nonnegative local minimum and monotone on each side, or one nonpositive local maximum and monotone on each side.

In any case, the function $w_i$ must have limits as $x \to \pm \infty$.   If both of these two limits are zero, we can easily conclude from the structure of $w_i$ that $w_i$ is identically zero, which is what we want.

Assume, to the contrary, that $w_i$ converges monotonically  to a positive number as $x \to +\infty$ (otherwise consider $-w_i$ instead). Since, in addition, $w_i$ is bounded, there exists a sequence $\{x_k\}\to+\infty, x_{k+1}-x_k\ge 1$, such that $ \partial_x w_i$ does not change signs on $[x_1,+\infty)$ and  $\partial_x w_i(x_k)\to 0$. By Lemma \ref{lem:v0 lower bound} we may assume that $c+|2\partial_x v_0/v_0|\le C$ in $\R$ for some positive constant $C$ independent of $k$. Then by integrating \eqref{eq:wi} from $x_k$ to $x_{k+1}$, we have
\[
\begin{split}
\int_{x_k}^{x_{k+1}} (\lambda_i-\lambda_0)w_i(x)\dd x&\le \partial_x w_i(x_{k+1})-\partial_x w_i(x_{k}) + C \int_{x_k}^{x_{k+1}} |\partial_x w_i(x)|\dd x\\
&=\partial_x w_i(x_{k+1})-\partial_x w_i(x_{k}) + C |\int_{x_k}^{x_{k+1}} \partial_x w_i(x)\dd x|\\
&=\partial_x w_i(x_{k+1})-\partial_x w_i(x_{k}) + C |w_i(x_{k+1})-w_i(x_k)|\\
&\to 0\quad\mbox{as }k\to\infty.
\end{split}
\]
This is in contradiction with the assumption that $w_i$ converges monotonically  to a positive number. Similarly, we can show that $w_i(x)$ converges to $0$ as $x\to-\infty$. Thus, we conclude $w_i\equiv 0$ for every $i=1,\dots,J$. It follows that $u(x,y)=v_0(x)\psi_0(y)$, $b(x)=v_0(x)\int_{\R^N}\psi_0(y)K(y)\dd y=\mu^{-1}(1-\lambda_0)v_0(x)$ where $\mu$ is the one in \eqref{eq:-infty}, and $v_0$ satisfies a classical Fisher-KPP equation
\[
-c \partial_x v_0 - \partial_{xx} v_0 =\mu^{-1}(1-\lambda_0)(\mu-v_0) v_0.
\]
Since $c<2\sqrt{1-\lambda_0}$, we have $v_0\equiv \mu$, and thus, $u\equiv V.$
\end{proof}

We remark that in \eqref{e:traveling-wave}, if $K=K(x,z)$ for $(x,z)\in \R\times\R^N$ and it satisfies \eqref{eq:condition-1} uniformly in $x\in\R$, then our proof still implies that the solution $u$ of \eqref{e:traveling-wave} has the separated structure $u(x,y) = v_0(x) \psi_0(y)$, where $v_0$ satisfies
\[
-c \partial_x v_0 - \partial_{xx} v_0 =\big(1-\lambda_0-a(x)v_0\big) v_0
\]
with
\[
a(x)=\int_{\R^N}\psi_0(z)K(x,z)\dd z.
\]

\subsection{Variations of the model}

\label{sec:othermodels}

Our proofs of existence and uniqueness for traveling fronts also apply to other models. The first example would be those with bounded traits and Neumann boundary conditions:
\begin{equation}\label{eq:neumann}
\begin{split}
cu_x - \Delta u + a(y) u &= \left(1 - \int_{\Omega} u(x,z)K(z)  \dd z \right) u,\quad(x,y)\in \R\times \Omega,\\
\frac{\partial u}{\partial \nu}(x,  y)&=0,\quad (x,y)\in \R\times\partial\Omega,
\end{split}
\end{equation}
where $\Omega$ is a bounded smooth domain in $\R^N$, $a, K$ are nonnegative bounded functions and the first eigenvalue $\lambda_0$ of the Neumann problem
\begin{align*}
-\Delta_y v + a v &= \lambda_0 v,\quad y\in \Omega,\\
\frac{\partial u}{\partial \nu}(y)&=0,\quad y\in \partial\Omega,
\end{align*}
satisfies $0< \lambda_0<1$. The application of our proofs to \eqref{eq:neumann} is quite straightforward. Therefore, we omit the details for \eqref{eq:neumann} and will focus on the second example below.

The mutations in \eqref{eq:non local interactions} 
in the space of trait $y$ may be modelled by a diffusion process other than Brownian motions. Indeed, it would make sense to think of mutations as a jump process in the trait variable. In the case of a simple $\alpha$-stable process, this leads us to models like
\begin{equation}\label{eq:mixed interactions}
u_t - \Delta_x u+ (-\Delta_y)^\sigma u + \alpha g(y) u = \left(1 - \int_{\R^N} u(t,x,z)K(z) \dd z \right) u,\quad (t,x,y)\in\R^+\times\R\times\R^N,
\end{equation}
where $\sigma\in (0,1)$ and $(-\Delta_y)^\sigma$ is the fractional Laplacian operator in $y$. Since the heat kernel of the fractional Laplacian is of polynomial decay, we assume the $K, g$ in \eqref{eq:mixed interactions} to satisfy
\begin{equation}\label{eq:condition-1-mix}
K\not\equiv 0,\ 0\le K(y)\le C_0 |y|^{\kappa_1}\quad\forall\ y\in\R^N\ \mbox{ with some fixed }\kappa_1\in [0,2\sigma),\ C_0>0,
\end{equation}
and $g$ is a H\"older continuous function satisfying 
\begin{equation}\label{eq:condition g-mix}
g(0)=0,\ 0<g\le C_0 |y|^{\kappa_1} \mbox{ in } \R^N\setminus\{0\}, \mbox{ and } \lim_{|y|\to+\infty}g(y)=+\infty.
\end{equation}
The traveling wave solutions of \eqref{eq:mixed interactions}, which are solutions of the type $u(x-ct,y)$, where $c\in\R$ is a constant, $u:\R^{N+1}\to\R$ satisfies
\begin{equation} \label{e:traveling-wave-mixed}  
-c u_x - \Delta_x u+ (-\Delta_y)^\sigma u + \alpha g(y) u = \left(1 - \int_{\R^N} u(x,z)K(z) \dd z \right) u
\end{equation}
such that \eqref{e:traveling-wave-infinity} holds. In addition, we require the traveling wave solutions $u$ has finite energy in the sense that $\|(-\Delta_y)^{\sigma/2} u(x,\cdot)\|^2_{L^2(\R^N)}$ and $\|\nabla_x u(x,\cdot)\|^2_{L^2(\R^N)}$ are locally integrable in $x$. 

To prove existence and uniqueness of traveling waves to \eqref{e:traveling-wave-mixed}, we only need propositions which are corresponding to Proposition \ref{prop:eigenfunction exp in y} and Proposition \ref{prop:exp in y}. We start with the analysis for the principal eigenvalue of the linear operator 
\[
\mathcal L_\sigma u(y)=(-\Delta)^\sigma u(y) +\alpha g(y)\quad\mbox{in }\R^N.
\]
Denote $H^\sigma(\R^N)$ be the standard fractional Sobolev space, and denote
\[
\mathcal H^\sigma(\R^N)=\{u\in H^\sigma(\R^N): \sqrt{g}u\in L^2(\R^N)\}
\]
with norm
\[
\|u\|_{\mathcal H^\sigma(\R^N)}=\left(\int |(-\Delta)^{\sigma/2}u|^2+gu^2\right)^{1/2}.
\]
As before, the embedding $\mathcal H^\sigma(\R^N)\hookrightarrow L^2(\R^N)$ is compact, and thus, Lemma \ref{l:spectral} holds for $\mathcal L_\sigma$ as well. Let
\[
\lambda=\inf\{\|u\|^2_{\mathcal H^\sigma(\R^N)}: u\in\mathcal H^\sigma(\R^N), \|u\|_{L^2(\R^N)}=1\},
\]
and for $R>0$,
\[
\lambda_R=\inf\{\|u\|^2_{\mathcal H^\sigma(\R^N)}: u\in\mathcal H^\sigma(\R^N), \|u\|_{L^2(\R^N)}=1, \ u\equiv 0\mbox{ in }\R^N\setminus B_R\}.
\]
Note that $\lambda$ is achieved by some positive function $\varphi\in\mathcal H^\sigma(\R^N)$ satisfying
\begin{equation}\label{eq:eigenfunction}
\begin{split}
(-\Delta)^\sigma \varphi+\alpha g \varphi&=\lambda \varphi\quad\mbox{in }\R^N
\end{split}
\end{equation}
and it is the principal eigenvalue for $\mathcal L_\sigma$. Also, $\lambda_R$ is achieved by some nonnegative function $0\not\equiv\varphi_R\in\mathcal H^\sigma$ satisfying
\[
\begin{split}
(-\Delta)^\sigma \varphi_R+\alpha g \varphi_R&=\lambda_R \varphi_R\quad\mbox{in }B_R,\\
 \varphi_R&=0\quad\mbox{in }\R^N\setminus B_R.
\end{split}
\]
Then we will have $\lambda_R$ converging to $\lambda$.
\begin{lemma}\label{prop:eigenfunction}
There holds
\[
\lim_{R\to\infty}\lambda_R=\lambda.
\]
\end{lemma}
\begin{proof}
First of all, we know that $\lambda_R$ is  non-increasing in $R$. Let
\[
 \lambda_0=\lim_{R\to\infty}\lambda_R. 
\]
Let $R_0$ be such that
\[
 \alpha g>\lambda_0+1\quad\mbox{in }\R^N\setminus B_{R_0}.
\]
Fixed $\varphi_R$ such that $\|\varphi_R\|_{L^2(\R^N)}=1$. We know that $\varphi_R>0$ in $B_R$. Then for $R>4R_0$, we have
\[
 \max_{B_{2R_0}} \varphi_R\le M,
\]
 where $M$ is independent of $R$. 

We claim that for all $R(>4R_0)$ sufficiently large such that $\lambda_R<\lambda_0+1$, there holds
\[
 \varphi_R\le M\quad\mbox{in }\R^N.
\]
Indeed, suppose there is  a large $R$ with $\max_{\R^N} (\varphi_{R}-M)>0$. Then the maximum is achieved at some point $\bar y\in B_{R}\setminus B_{2R_0}$. Thus, we have
\[
 (-\Delta)^\sigma (\varphi_{R_0}-M)(\bar y)>0.
\]
This implies
\[
 0> (\lambda_R-\alpha g(\bar x))\varphi_{R}(\bar y)=(-\Delta)^\sigma \varphi_{R}(\bar y)>0,
\]
which is a contradiction. 

Therefore, $\varphi_R$ is uniformly bounded. By the H\"older estimates (see, e.g., Proposition 2.9 in \cite{silvestre}), subject to a subsequence, $\varphi_R$ converges locally uniformly to a bounded nonnegative continuous function $\varphi$. Since $\varphi_R$ is also bounded in $\mathcal H^\sigma(\R^N)$, we have $\varphi\in\mathcal H^\sigma(\R^N)$ satisfies $\|\varphi\|_{L^2(\R^N)}=1$, and is a solution of
\[
 (-\Delta)^\sigma\varphi+\alpha g\varphi=\lambda_0\varphi.
\]
Hence, $\varphi$ is positive  in $\R^N$, and therefore, $\lambda=\lambda_0$. 
\end{proof}

The first eigenfunction in \eqref{eq:eigenfunction} decays at polynomial rates. This follows by very classical methods. It is essentially the same decay as the Bessel potential or the fractional heat kernel. See for example the appendix in \cite{Felmer}.

\begin{prop}\label{prop:eigenfunction estimate-polynomial}
Suppose $\varphi\in\mathcal H^\sigma(\R^N)$ is a nonnegative solution of \eqref{eq:eigenfunction}, then
\[
 \varphi(y)\le \frac{C}{|y|^{N+2\sigma}}\quad\mbox{in }\R^N.
\]
\end{prop}
\begin{proof}
 Let $R_0$ be such that
\[
 \alpha g>\lambda+1\quad\mbox{in }\R^N\setminus B_{R_0}.
\]
Let $G$ be a Green function satisfying 
\[
 ((-\Delta_y)^\sigma+1)G=\delta_0.
\]
We know that $G$ is positive, radial, strictly decreasing in $|y|$, and satisfies
\[
G(y)\le \frac{C}{|y|^{N+2\sigma}}\quad\mbox{for }|y|\ge 1,
\]
where $C$ is a positive constant depending only on $N$ and $\sigma$. Let  $\eta(y)=|y|^\sigma$.
Then there exists $R_1\in [R_0,+\infty)$ such that
\[
 (-\Delta)^\sigma\eta(y)+\eta(y)\ge 0\quad\mbox{for all }|y|\ge R_1.
\]
We know from the proof of Lemma \ref{prop:eigenfunction} that $\varphi$ is a bounded function. Therefore, we can choose $M$ large enough to have 
\[
 MG(R_1)\ge\varphi(y) \quad\mbox{for all }|y|\le R_1.
\]
For every $\eps\in (0,1)$, we claim
\[
\varphi\le MG+\eps\eta \quad\mbox{in }\R^N.
\]
Suppose the contrary: $\max_{\R^N} (\varphi-MG-\eps\eta)>0 $ is achieved at some point $\bar y\in \R^N\setminus \overline B_{R_1}$. Then
\[
\begin{split}
 0&< (-\Delta)^\sigma(\varphi-MG-\eps\eta)(\bar y)+(\varphi-MG-\eps\eta)(\bar y)\\
&=(\lambda+1-\alpha g(\bar y))\varphi(\bar y)-\eps  \Big((-\Delta)^\sigma\eta(\bar y)+\eta(\bar y)\Big)\\
&\le 0,
\end{split}
\]
which is a contradiction. Therefore, by sending $\eps \to 0$, we have
\[
 \varphi(y)\le MG(y)\le \frac{C}{|y|^{N+2\sigma}}\quad\mbox{in }\R^N.
\]
This finishes the proof.
\end{proof}

\begin{prop}\label{prop:decay mixed}
If $u$ is a nonnegative bounded solution of the traveling wave equation \eqref{e:traveling-wave-mixed},
then we have
\[
 u(x,y)\le \frac{C}{|y|^{N+2\sigma}}\quad\mbox{in }\R^{N+1}.
\]
\end{prop}

\begin{proof}
Let $\eta$ and $G$ be as the one in the proof of Proposition \ref{prop:eigenfunction estimate-polynomial}. For every $\eps\in (0,1)$,  we have
\[
 |\frac{(-\Delta)^\sigma (G+\eps\eta)(y)}{(G+\eps\eta)(y)}|\le  \frac{G+\eps|(-\Delta)^\sigma \eta(y)|}{G(y)+\eps\eta(y)}\le 1 + \frac{|(-\Delta)^\sigma \eta(y)|}{\eta(y)}\le c_0\quad\mbox{in }|y|\ge 1,
\]
for some positive constant $c_0$. 
Let $R_0>1$ be such that
\[
 \alpha g(y)>3|c|+c_0+3\quad\mbox{in }|y|\ge R_0.
\]
Since $u$ is a bounded function, we can choose $M$ large enough so that
\[
 MG(R_0)\ge u(x,y)\quad\mbox{in }|y|\le R_0.
\]
Let $w(x,y)=M(1+\eps x \arctan x)(G(y)+\eps \eta(y))$.
We claim
\[
 u(x,y)\le w(x,y)\quad\mbox{in }\R^{N+1}.
\]
If not, then $\max_{\R^{N+1}}(u-w)>0$ is achieved at some point $(\bar x, \bar y)$. It follows that $|\bar y|> R_0.$ Therefore,
\[
 0<-c(u-w)_x-\Delta_x (u-w)+(-\Delta_y)^\sigma (u-w)+\alpha g (u-w)+u-w
\]
On the other hand, we have
\[
\begin{split}
&-cw_x-\Delta_x w+(-\Delta_y)^\sigma w+\alpha g w-w\\
&=M\Big(\eps\big(-c(\arctan x+\frac{x}{1+x^2})-\frac{2}{(1+x^2)^2}\big)\\
&\quad\quad +(\alpha g(y)-\frac{(-\Delta)^\sigma (G+\eps\eta)(y)}{(G+\eps\eta)(y)}-1)(1+\eps x \arctan x)\Big)(G(y)+\eps \eta(y))\\
&\ge C\left(-3|c|-2+\alpha g(y)-c_0-1\right)(G(y)+\eps \eta(y))\\
&> 0.
\end{split}
\]
Therefore, at $(\bar x, \bar y)$, we have
\[
-c(u-w)_x-\Delta_x (u-w)+(-\Delta_y)^\sigma (u-w)+\alpha g (u-w)+u-w<0,
\]
which is a contradiction. Thus, our claim holds. By sending $\eps \to 0,$ we derive\[
 u(x,y)\le \frac{C}{|y|^{N+2\sigma}}\quad\mbox{in }\R^{N+1}.
\]
This completes the proof.
\end{proof}

From Proposition \ref{prop:decay mixed} we get the decomposition \eqref{e:decomposition}. Then for those traveling wave solutions with finite energy, we can split \eqref{e:traveling-wave-mixed} into a sequence of equation as in \eqref{e:equations-for-vi}.  Owing to the assumptions on $g$ and $K$, the terms $gu$ and $Ku$ are decaying faster than $|y|^{-N}$. Now we can conclude from the proof of Theorem \ref{t:intro-stationary} that for $0<\alpha<\tilde \alpha$, where $\tilde \alpha$ is uniquely determined by $\lambda(\tilde \alpha)=1$ in \eqref{eq:eigenfunction}, we have

\begin{thm} \label{t:intro-stationary-mixed}
There exists a positive number $c^*$ so that
\begin{itemize}
\item If $0\le c < c^*$, there exists only one positive bounded solution $u$ of \eqref{e:traveling-wave-mixed} with finite energy. Moreover, the solution is constant in $x$.
\item If $c \geq c^*$, there exists a unique non negative bounded solution $u$ of \eqref{e:traveling-wave-mixed} with finite energy such that \eqref{e:traveling-wave-infinity} holds.
\end{itemize}
\end{thm}

\section{Asymptotic speed of propagation}\label{sec:asp}

We consider the Cauchy problem \eqref{eq:non local interactions} with $u(0,x,y)= u_0(x,y)$, where $u_0$ is smooth, with compact support in $\R^{m+N}$,  
$u_0\ge0$, and $u_0\not\equiv 0$. Let $C_0$ be a positive constant such that 
\begin{equation}\label{eq: bound u0}
0\le u_0\le C_0 \psi_0\quad\mbox{in }\R^{m+N},
\end{equation}
where $\psi_0$ is the first eigenfunction in Lemma \ref{l:spectral}. The function $e^{(1-\lambda_0)t} \psi_0$ is a solution of the linear equation
\begin{equation}\label{eq:linear}
\partial_t \psi-\Delta \psi+\alpha g \psi=\psi,
\end{equation}
where $\lambda_0$ is the first eigenvalue in Lemma \ref{l:spectral}. By the comparison principle, standard parabolic equation estimates and fixed point arguments, there exists a unique solution $u$ of \eqref{eq:non local interactions} such that $u(0,x,y)= u_0(x,y)$ for all time $0<t<\infty$, $u$ is smooth in $(0,+\infty)\times\R^m\times\R^N$ and satisfies
\begin{equation}\label{eq:bounded by t}
 0\le u(t,x,y)\le C_0 e^{(1-\lambda_0)t} \psi_0(y)\quad\mbox{for all }(t,x,y)\in (0,+\infty)\times\R^m\times\R^N,
\end{equation}
where $C_0$ is the constant in \eqref{eq: bound u0}.

We are interested in the long time behavior of the solution $u$ as $t\to\infty$. In this section, we are going to prove Theorem \ref{t:into-asp} on the asymptotic speed of propagation, which is the main result of this section. To prove Theorem \ref{t:into-asp}, we proceed in two steps. We first prove the weaker version in Theorem \ref{thm:asp} below. It consists in showing that for large time, for every $y$, the solution $u(t,x,y)$ is bounded from below by a positive constant on the sets  $\{ |x|\leq ct\}$ when $c<c^*$. Then, in Theorem \ref{convergence} we make use of the decomposition as in Section \ref{sec:reduction} to obtain the more precise behavior that $u$ converges to $V(y)$ on these sets. This yields Theorem \ref{t:into-asp}.

Similar spreading rates for solutions of the local equation \eqref{eq:BC} were obtained in \cite{BC}.  As usual, the bound \eqref{eq:up bound speed} in Theorem \ref{thm:asp} for $c>c^*$ follows immediately from comparing the solution of \eqref{eq:non local interactions} and the solution of the linear equation \eqref{eq:linear}.

However, because of the lack of general comparison principles, the proof of the bound \eqref{eq:lower bound speed} in Theorem \ref{thm:asp} for $c<c^*$ is quite different from that in \cite{BC}. In this step, we shall adapt some compactness arguments used by Hamel and Ryzhik in \cite{HR}. The general idea is the following. If $u(t,x,y_0)$ is small for $|x|< c^*t$ and some point $y_0$, then $\int_{\R^N} u(t,x,y)K(y)\dd y$ will be small. Hence, the behavior of $u$ should be similar to that of the solution of the linear equation \eqref{eq:linear}, which, however, is not small for $|x|< c^*t$.

To employ the compactness arguments, we first need to establish a uniform upper bound estimate for $u$, which, unlike \eqref{eq:bounded by t}, will be independent of the time $t$.

\subsection{A priori estimates}

To obtain the uniform upper bound of $u$, in addition to \eqref{eq:condition-1}, we assume
\begin{equation}\label{eq:assumption-2}
K(y)\ge K_1 \mbox{ for } |y|\le R_0+2,
\end{equation}
where $K_1$ is a positive constant, and $R_0$ is chosen such that
\begin{equation}\label{eq:R0}
\alpha g(y)\ge 1\quad\mbox{for all } |y|\ge R_0.
\end{equation}

As an intermediate step, we show the following auxiliary uniform estimate.
\begin{lemma}\label{lem:bound-1}
There exists a positive constant $M_1$ depending only on $C_0$, $K_1$, $\alpha$ and $g$ such that 
\[
\int_{B_1} u(t,x,y+s)\dd s\le M_1\quad\mbox{for all }(t,x,y)\in (0,+\infty)\times\R^m\times\R^N,
\]
where $B_1$ is the unit ball centered at the origin in $\R^N$.
\end{lemma}

\begin{proof}
Let
\[
 v(t,x,y)=\int_{B_1} u(t,x,y+s)\dd s=\int_{B_1(y)} u(t,x,s)\dd s,
\]
where $B_1(y)$ is the ball in $\R^N$ with radius $1$ and center $y$. Then
\begin{equation}\label{eq:vn integral}
 v_t-\Delta v+\alpha\int_{B_1(y)} g(s) u(t,x,s)\dd s=v\left(1-\int_{\R^N}u(t,x,z)K(z)\dd z\right).
\end{equation}
By \eqref{eq:bounded by t}, we have
\[
0\le v(t,x,y)\le |B_1|C_0 e^{(1-\lambda_0)t}\|\psi_0\|_{L^\infty(\R^N)}.
\]
Let
\[
M_1=\max(1/K_1, |B_1|C_0 e^{1-\lambda_0}\|\psi_0\|_{L^\infty(\R^N)})+1.
\]
We are going to show 
\[
v(t,x,y)< M_1\quad\mbox{for all }(t,x,y)\in (0,+\infty)\times\R^m\times\R^N.
\]
Suppose there is $t_0$ such that
\[
 \|v(t_0,\cdot)\|_{L^\infty(\R^{m+N})}=M_1,\quad\mbox{and}\quad \|v(t,\cdot)\|_{L^\infty(\R^{m+N})}<M_1\quad\mbox{for }t<t_0.
\]
Then $t_0\ge 1$, and there exists a sequence $\{(x_n,y_n)\}$ such that $v(t_0,x_n,y_n)\to M_1$ as $n\to\infty$. From \eqref{eq:bounded by t} we infer that $\{y_n\}$ is a bounded sequence.  We define the translations (in $x$)
\[
 u_n(t,x,y)=u(t,x+x_n,y)\quad\mbox{and}\quad  v_n(t,x,y)=v(t,x+x_n,y),
\]
which also satisfy \eqref{eq:non local interactions} and \eqref{eq:vn integral}, respectively. By \eqref{eq:bounded by t}, parabolic equation estimates and dominated convergence theorem, up to a subsequence, $y_n\to y_\infty$, $\{u_n\}$ converges locally uniformly to $u_\infty$ which satisfies \eqref{eq:non local interactions}, and $\{v_n\}$ converges locally uniformly to $v_\infty$ which satisfies the equation \eqref{eq:vn integral} associated with $u_\infty$. Moreover,
\[
 0\le v_\infty(t,x,y)\le M_1 \quad\mbox{ in }(0,t_0)\times\R^m\times\R^{N}, \quad v_\infty(t_0,0,y_\infty)=M_1,
\]
and thus
\[
 \partial_t v_\infty(t_0,0,y_\infty)\ge 0,\quad \Delta v_\infty(t_0,0,y_\infty)\le 0.
\]
This implies
\[
\alpha\int_{B_1(y_\infty)} g(s) u_\infty(t_0,0,s)\dd s \le v_\infty(t_0,0,y_\infty)\left(1-\int_{\R^N}u_\infty(t_0,0,z)K(z)\dd z\right).
\]
Hence,
\[
\alpha\int_{B_1(y_\infty)} g(s) u_\infty(t_0,0,s)\dd s < v_\infty(t_0,0,y_\infty)=\int_{B_1(y_\infty)} u_\infty(t_0,0,s)\dd s
\]
and
\[
\int_{\R^N}u_\infty(t_0,0,z)K(z)\dd z\le 1.
\]
Thus, by the choice of $R_0$ in \eqref{eq:R0}, we have
\[
 |y_\infty|\le R_0+1.
\]
From the assumption \eqref{eq:assumption-2},
we derive
\[
K_1M_1=K_1 v_\infty(t_0,0,y_\infty)=K_1 \int_{B_1(y_\infty)} u_\infty(t_0,0,z)\dd z \le \int_{B_1(y_\infty)} u_\infty(t_0,0,z)K(z)\dd z\le 1.
\]
This contradicts the choice of $M_1$. Thus, we proved that no such $t_0$ exists, from which the lemma follows.
\end{proof}

We can now derive a uniform bound on $u$ independently of the time $t$.

\begin{lemma}\label{lem:uniform bound}
There exists a positive constant $M_2$ depending only on $C_0$, $K_1$, $\alpha$ and $g$ such that  
\[
 0\le u(t,x,y)\le M_2\quad\mbox{in }(0,+\infty)\times\R^m\times\R^{N}.
\]
\end{lemma}

\begin{proof}
Let $M_2$ be a sufficiently large constant to be fixed in the proof. 
Suppose there exists $t_0>0$ such that
\[
 \|u(t_0,\cdot)\|_{L^\infty(\R^{m+N})}=M_2\quad\mbox{ and } \|u(t,\cdot)\|_{L^\infty(\R^{m+N})}<M_2\quad\mbox{ for }t<t_0.
\]
By \eqref{eq:bounded by t} we can choose $M_2$ large enough so that $t_0\ge \sqrt{2R_0} $, where $R_0$ is the constant in \eqref{eq:R0}.
There exists a sequence $\{(x_n,y_n)\}$ for which $u(t_0,x_n,y_n)\to M_2$ as $n\to\infty$. We reason as in the proof of the preceding lemma. From \eqref{eq:bounded by t} we infer that $\{y_n\}$ is a bounded sequence.  As before, we define the translation (in $x$)
\[
 u_n(t,x,y)=u(t,x+x_n,y),
\]
which also satisfies \eqref{eq:non local interactions}. By \eqref{eq:bounded by t}, parabolic equation estimates and dominated convergence theorem, up to a subsequence, $y_n\to y_\infty$, and $\{u_n\}$ converges locally uniformly to $u_\infty$ which satisfies \eqref{eq:non local interactions}. Moreover,
\[
 0\le u_\infty(t,x,y)\le M_2\quad\mbox{ in }(0,t_0)\times\R^m\times\R^{N}, \quad u_\infty(t_0,0,y_\infty)=M_2,
\]
and thus
\[
 \partial_t u_\infty(t_0,0,y_\infty)\ge 0,\quad \Delta u_\infty(t_0,0,y_\infty)\le 0.
\]
This implies
\[
 \alpha g(y_\infty) M_2 \le (1-\int_{\R^N} u_\infty(t_0,0,z)K(z)\dd z)M_2.
\]
Thus,
\[
 |y_\infty|\le R_0\quad\mbox{and}\quad \int_{\R^N} u_\infty(t_0,0,z)K(z)\dd z\le 1.
\]
Let $\Omega=(t_0-\sqrt{2R_0}, t_0]\times \{x\in\R^m:|x|\le 2R_0\}\times \{y\in\R^N:|y|\le 2R_0\}$. The limit $u_\infty$ satisfies
\[
 \partial_t u_\infty-\Delta u_\infty+\alpha g(y) u_\infty\le u_\infty\quad\mbox{in }\Omega.
\]
By the local maximum principle in Lemma \ref{lem:local max principle}, we have
\[
 M_2=u_\infty(t_0,0,y_\infty)\le C\int_\Omega u_\infty(t,x,y)\dd t\dd x \dd y,
\]
where $C>0$ depends only on $\alpha$ and $g$. By Lemma \ref{lem:bound-1}, we have
\[
 \int_\Omega u_\infty(t,x,y)\dd t\dd x \dd y\le \tilde C M_1.
\]
where $\tilde C>0$ depends only on  $R_0$. Thus $M_2\le C\tilde C M_1$. This is a contradiction if we choose $M_2$ large enough. Hence, we proved that no such $t_0$ exists, from which the lemma follows.
\end{proof}

As a consequence we can show the uniformly exponential decay in $y$ of  $u$, independently of the time $t$.

\begin{lemma}\label{lem:uniform weighted bound}
For every $\gamma>0$ there exists a positive constant $M$ depending only on $C_0,K_1,\alpha, g$ and $\gamma$ such that
\[
 0\le u(t,x,y)\le M e^{-\gamma |y|}\quad\mbox{in }(0,+\infty)\times\R^m\times\R^N.
\]
\end{lemma}
\begin{proof}
Let $w(x,y)=M(1+\eps |x|^2)(e^{-\gamma |y|}+\eps e^{\gamma |y|})$ for some $M>0$ and $\eps\in (0,1)$. 
A direct computation shows that 
\[
-\Delta w+\alpha g w-w\ge 0
\]
for $(x,y)\in \Omega:=\R^m\times \{y: |y|\ge R_0\}$, where $R_0\ge (N-1)\gamma$ 
is chosen so that $g(y)\ge\alpha^{-1}(\gamma^2+2m+2)$ for all $|y|\ge R_0$. By Proposition \ref{prop:eigenfunction exp in y}, 
we can choose $M$ large so that $Me^{-\gamma y}\ge C_0\psi_0(y)\ge u_0(x,y)$ for $(x,y)\in\Omega$, and $Me^{- \gamma R_0}\ge M_2$, where $M_2$ is the one in Lemma \ref{lem:uniform bound}. Since
 \begin{equation}\label{eq:u subsolution}
 u_t-\Delta u+\alpha g(y) u\le u,
\end{equation}
it follows from the comparison principle that
\[
u\le w\quad\mbox{in }\quad\mbox{in }(0,+\infty)\times\R^{m+N}.
\]
The conclusion follows by sending $\eps\to 0$.
\end{proof}

\subsection{Uniqueness of stationary solutions}
In this section, we consider nonnegative bounded stationary solutions $u=u(x,y)$ of \eqref{eq:non local interactions}, that is, solutions
 of the equation:
\begin{equation}\label{stat}
 - \Delta u + \alpha g(y) u = \left(1 - \int_{\R^N} K(z) u(x,z) \dd z \right) u,\quad (x,y)\in\R^m\times\R^N.
\end{equation}
By using the techniques in Section \ref{sec:tr}, we are able to show the uniqueness of the stationary solution. 

\begin{thm} \label{Liouville-stat} For $0<\alpha < \bar\alpha$, if $u\not\equiv 0$ is a non-negative bounded solution of \eqref{stat}, then $u(x,y)\equiv V(y)$, where $V$ is defined in \eqref{eq:steady solution}.
When $\alpha >�\bar\alpha$, the unique non-negative solution of \eqref{stat} is identically zero.
\end{thm}
\begin{proof}
First of all, it follows from the same proof as for Lemma \ref{lem:uniform weighted bound} that for every $\gamma>0$ there exists a positive constant $M$ such that
\begin{equation}\label{eq:stationary decay}
 0\le u(x,y)\le M e^{-\gamma |y|}\quad\mbox{in }\R^m\times\R^N.
\end{equation}
Therefore, as in Section \ref{sec:tr},  we can write $u$ as
\begin{equation}\label{e:decomposition-m}
u(x,y) = \sum_{i=0}^\infty v_i(x) \psi_i(y),
\end{equation}
where the $\psi_i$ are from the orthonormal basis in Lemma \ref{l:spectral} and $v_i(x)=\int_\R u(x,z)\psi_i(z)\! \dd z\in L^\infty(\R^m)$.  
Moreover, the equation \eqref{stat} splits into a sequence of equations for each $v_i$:
\begin{equation} \label{e:equations-for-vi-m}  
- \Delta v_i  = (1-\lambda_i-b(x)) v_i,
\end{equation}
where $ b(x) = \int_{\R^N} u(x,z) K(z)\dd z$. 

Suppose first that $\alpha<\bar\alpha$. By arguments similar to the proof of Lemma \ref{lem:v0 lower bound} we get 
\begin{equation}\label{eq:lower bound v0}
\inf_{x\in\R^m} v_0(x)>0.
\end{equation}
Indeed, suppose there exists a sequence $\{x_k\}$, $|x_k|\to\infty$, for which
$
v_0(x_k)\to 0.
$
Then, by exactly the same proof as for \eqref{eq:aux for v0}, for every $L>0$, we have
\[
\lim_{k\to\infty}\sup_{|x-x_k|\le L} v_0(x)=0.
\]
Since $\alpha<\bar\alpha$, we have $\lambda_0<1$. Choose $\delta>0$ such that $\lambda_0<1-\delta$. 
Let $\beta>0$ and $\varphi_0(x)$ be the first eigenvalue and first eigenfunction of the Dirichlet problem in the unit ball of $\R^m$ with the normalization of unit $L^\infty$ norm. That is,
\begin{equation}\label{eq:dirichilet unit}
\begin{cases}
-\Delta_{x}\varphi_0=\beta\varphi_0\quad\mbox{on }B_1:=\{x\in\R^m: |x|<1\},\\
\varphi_0>0\quad\mbox{in }B_1,\quad \varphi_0=0\quad\mbox{on }\partial B_1, \quad \|\varphi_0\|_{L^\infty(B_1)}=1.
\end{cases}
\end{equation}
Let $L=\sqrt{\frac{\beta}{1-\lambda_0-\delta}}$ and $\phi_k(x)=\varphi_0((x-x_k)/L)$, then it satisfies
\[
- \Delta\phi_k+ (\lambda_0+\delta-1) \phi_k=0.
\]
There exists $\eps>0$ such that $\eps\phi_k$ touches $v_0$ from below at some point $\bar x_k\in \{x:|x-x_k|< L\}$. Then by evaluating the equation
\[
- \Delta (v_0-\eps\phi_k)+ (\lambda_0-1) (v_0-\eps\phi_k)=-b v_0+\delta\eps\phi_k\quad\mbox{at }\bar x_k,
\]
we have $b(\bar x_k)\ge\delta$. The rest is identical to the proof of Lemma \ref{lem:v0 lower bound}. This proves \eqref{eq:lower bound v0}.

Let  $i\geq 1$ be fixed. We set
\[
w_i:= \frac{v_i}{v_0}.
\]
From \eqref{e:equations-for-vi-m} we see that $w_i$ satisfies
\begin{equation}\label{eq:w-i-m}
-\Delta w_i - 2 \frac{\nabla v_0}{v_0} \cdot \nabla w_i\, = \, (\lambda_0 - \lambda_i) w_i.
\end{equation}
We know from \eqref{eq:lower bound v0} that $w_i$ is bounded in $\R^m$ and the above equation has bounded coefficients. Suppose $w_i$ is not identically zero. We can assume that it is positive somewhere. Then
\[
0 < \sup_{\R^m} w_i < \infty.
\]
If this supremum is reached at a point $\bar x\in\R^m$, then, since $\lambda_0 < \lambda_i$, we get $\Delta w_i (\bar x) >0$ from \eqref{eq:w-i-m}, which is absurd. So let us assume that for some sequence 
$x_j\in\R^m$, with $|x_j|\to\infty$, we have
\[
\lim_{j\to\infty} w_i(x_j)=  \sup_{\R^m} w >0.
\]
Let us now set $w_{i,j}(x):= w_i(x_j+x)$ and $v_{0,j}(x) : = v_0(x_j+x)$. From \eqref{eq:w-i-m} we get
\[
-\Delta w_{i,j} - 2 \frac{\nabla v_{0,j}}{v_{0,j}} \cdot \nabla w_{i,j}\, = \, (\lambda_0 - \lambda_i) w_{i,j}.
\]
Since $v_0$ is bounded from below and satisfies \eqref{e:equations-for-vi-m}, by elliptic regularity estimates, we can strike out a subsequence, which is still denoted by $j$, such that
\[
v_{0,j} \To v_{0,\infty}, \quad w_{i,j}\To w_{i,\infty}, \quad \inf_{\R^m} v_{0,\infty} >0.
\]
Moreover, we have $w_{i,\infty} \leq \sup_{\R^m} w_i$, $w_{i,\infty} (0) = \sup_{\R^m} w_i$ 
whence $w_{i,\infty} (0) = \sup_{\R^m} w_{i,\infty}$ and $w_{i,\infty}$ satisfies the equation 
\[
-\Delta w_{i,\infty} - 2 \frac{\nabla v_{0,\infty}}{v_{0,\infty}} \cdot \nabla w_{i,\infty}\, = \, (\lambda_0 - \lambda_i) w_{i,\infty}.
\]
We reach a contradiction by analyzing this equation at $0$.

This proves $v_i\equiv 0$ for all $i\ge 1$. Therefore, every nonnegative bounded solution of \eqref{stat} satisfies
\[
u(x,y)=v_0(x)\psi_0(y).
\]
Hence, $b(x)=v_0(x)\int_{\R^N}\psi_0(y)K(y)\dd y=\mu^{-1}(1-\lambda_0)v_0(x)$ where $\mu$ is the one in \eqref{eq:-infty}, and $v_0$ satisfies a classical Fisher-KPP equation
\[
- \Delta v_0 =\mu^{-1}(1-\lambda_0)(\mu-v_0) v_0.
\]
We have $v_0\equiv \mu$, and thus, $u\equiv V.$ We remark that this translation and compactness proof can also be used to prove Theorem \ref{thm:liouville with c}.

Suppose now that $\alpha>\bar\alpha$. We want to show that $v_i\equiv 0$ for all $i\ge 0$. We can do the above translation and compactness arguments for \eqref{e:equations-for-vi-m} directly, since $1-\lambda_i-b(x)\le 1-\lambda_0<0$. Suppose $v_i\not\equiv 0$ for some $i$. We can assume that it is positive somewhere, and then 
$
0 < \sup_{\R^m} v_i < \infty.
$
By the equation \eqref{e:equations-for-vi-m}, this positive supremum cannot be achieved at any point.
So there exists some sequence $x_j\in\R^m$, with $|x_j|\to\infty$, for which
\[
\lim_{j\to\infty} v_i(x_j)=  \sup_{\R^m} v_i >0.
\]
Then we do a translation $v_{i,j}(x)=v_i(x_j+x)$, $b_j(x)=b(x_j+x)$ and $u_j(x,y)=u(x_j+x, y)$. By elliptic regularity estimates, after extraction of a subsequence, we can assume that $v_{i,j}$ and $b_j$ are locally uniformly convergent to $v_{i,\infty}$ and $b_\infty$, respectively. These limits satisfy
\[
- \Delta v_{i,\infty}  = (1-\lambda_i-b_\infty(x)) v_{i,\infty}.
\]
Moreover, $v_{i,\infty}\le \sup_{\R^m} v_i$ in $\R^m$, $v_{i,\infty}(0)=\sup_{\R^m} v_i$. We reach a contradiction by evaluating the above equation at $0$. Therefore, in the case of $\alpha>\bar\alpha$, every bounded nonnegative solution of \eqref{stat} has to be identically zero.
\end{proof}

\subsection{Asymptotic speed of propagation}

In this section we prove the following long time behavior properties for solutions of the Cauchy problem \eqref{eq:non local interactions} with compactly supported nonnegative initial data. This is a weaker version of the results stated in Theorem~\ref{t:into-asp} that we require as a first step in proving the stronger version.

\begin{thm}\label{thm:asp}
Assume conditions \eqref{eq:condition-1}, \eqref{eq:assumption-2} and \eqref{eq:condition g}. Consider the 
solution $u$ of \eqref{eq:non local interactions} 
with $u(\cdot, 0)=u_0\,$ smooth, having compact support in $\R^{m+N}$, $u_0\ge 0$, and $u_0 \not\equiv 0$. 
\begin{enumerate}
\item[(i): ] if $\alpha>\bar\alpha$, then $u(t,x,y)\to 0$ exponentially in $t$, uniformly in $(x,y)$.

\item[(ii):] if $0<\alpha<\bar\alpha$, then, for every $y\in\R^N$, 
\begin{equation}\label{eq:lower bound speed}
  \liminf_{t\to+\infty}\Big(\min_{|x|\le ct}u(t,x,y)\Big)>0\quad\mbox{ for all }0\le c<c^*, 
\end{equation}
and
\begin{equation}\label{eq:up bound speed}
\lim_{t\to+\infty}\Big(\sup_{|x|\ge ct, y\in\R^N}u(t,x,y)\Big)=0\quad\mbox{ for all }c>c^*.
\end{equation}

\end{enumerate}
\end{thm}
\begin{proof}
The conclusion in $(i)$ immediately follows from \eqref{eq:bounded by t} since   $\lambda_0>1$ when $\alpha>\bar\alpha$.

To prove \eqref{eq:up bound speed} we shall use the exponential solutions $\psi_e(t, x,y)=M_3e^{-\frac{c^*}{2}(x\cdot e-c^* t)}V(y)$, $e\in\mathbb S^{m-1}$, which are solutions of \eqref{eq:linear}.  
From \eqref{eq:u subsolution} and by the comparison principle, if we choose $M_3>0$ large, we have
\[
 u\le\psi_e, \quad \mbox{for all}\quad e\in\mathbb S^{m-1}.
\]
By minimizing over $e$ (for each $x$) we derive:
\[
u(t,x,y) \leq M_3 e^{-\frac{c^*}{2}(c-c^* ) t)}V(y), \quad \mbox{for all}\quad |x|\geq ct, y\in \R^N.
\]
Therefore, for $c>c^*$,
\[
\lim_{t\to\infty} \left( \sup_{ |x| \geq ct, y\in\R^N}  u(t, x,y)   \right) = 0.
\] 

To prove \eqref{eq:lower bound speed}, we use some compactness arguments as in \cite{HR}. We argue by contradiction. Suppose that there are $c<c^*$, $y_0\in\R^N$ and a sequence $\{(t_n,x_n)\}$ such that
\[
 \begin{cases}
  |x_n|\le c t_n\quad\mbox{for all }n\in\mathbb N,\\
   t_n\to\infty\quad\mbox{and } u(t_n,x_n,y_0)\to 0\quad\mbox{as }n\to\infty.
 \end{cases}
\]
We may assume that $c_n:=|x_n|/t_n\to c_\infty\in [0,c]$ as $n\to\infty$. We let $e_n=x_n/|x_n|\in \mathbb S^{m-1}$ 
 (if $x_n=0$, we let $e_n$ be the north pole of $\mathbb S^{m-1}$) and assume $e_n\to e_\infty$ as $n\to\infty$.
 
For each $n$ and $(t,x)\in (-t_n,+\infty)\times\R^m$, we define the translation of $u$ in $(t,x)$
\[
 u_n(t,x,y)=u(t+t_n,x+x_n,y).
\]
By Lemma \ref{lem:uniform weighted bound}, standard parabolic equation estimates, and dominated convergence theorem, there exists a subsequence of $\{u_n\}$, which we still denote by $\{u_n\}$, such that $u_n$ is locally uniformly convergent to $U$ satisfying
\begin{equation}
 \partial_t U - \Delta U + \alpha g U=\left(1 - \int_{\R^N} U(t, x,z)K(z) \dd z \right) U\quad\mbox{in }\R\times \R^m\times\R^{N}.
\end{equation}
Moreover,
\[
U(0,0,y_0)=0,\quad U\ge 0\quad\mbox{in }\R\times \R^m\times \R^{N}.
\]
By the strong maximum principle, $U\equiv 0$ in $(-\infty,0]\times \R^{m+N}$. 
Consequently, by the comparison principle, we also have $ U\equiv 0$ in $[0,+\infty)\times \R^{m+N}$, and thus, 
\[
 U\equiv 0\quad\mbox{in }\R\times \R^{m+N}.
\]
Let
\[
v_n(t,x,y)=u_n(t,x+c_n t e_n, y)=u(t+t_n, x+c_n(t+t_n)e_n, y).
\]
Then
\begin{equation}\label{eq:for vn}
\partial_t v_n - \Delta v_n-c_n e_n\cdot \nabla_x v_n + \alpha g v_n=\left(1 - \int_{\R^N} v_n(t, x,z)K(z) \dd z \right) v_n,
\end{equation}
Since $c_n$ is bounded, $\{v_n\}$ also converges  locally uniformly  to $0$ in $\R\times\R^{m+N}$. By Lemma \ref{lem:uniform weighted bound}, we see that $\int_{\R^N} v_n(t,x,y)K(y)\dd y$ converges to $0$ locally uniformly  as well.

Since $c<c^*=2\sqrt{1-\lambda_0}$, we can choose $\delta>0$ so as to have
\[
|c_n|\le c<2\sqrt{1-\lambda_0-2\delta}.
\]
Now let us use the property that $\lambda_0$ is the limit of the principal eigenvalue $\lambda^R$ of the Dirichlet problem in $B_R\subset\R^N$ as $R\to\infty$ (see \cite{BR06} for more details). That is:
\[
\begin{cases}
-\Delta_{y}\psi^R+\alpha g \psi^R=\lambda^R \psi^R\quad\mbox{on }B_R:=\{y\in\R^N: |y|<R\},\\
\psi^R>0\quad\mbox{in }B_R,\quad \psi^R=0\quad\mbox{on }\partial B_R, \quad \|\psi^R\|_{L^\infty}=1.
\end{cases}
\]
More precisely, $\lambda^R>\lambda_0$ and $\lambda^R\to\lambda_0$ as $R\to\infty$. We can choose $R$ large enough to have $\lambda_0<\lambda^R<\lambda_0+\delta$ and $|y_0|\le R/2$. Then $c<2\sqrt{1-\lambda^R-\delta}$. 

Let $\beta>0$ and $\varphi_0(x)$ be the elements in \eqref{eq:dirichilet unit}. 
Let $L=\sqrt{\frac{4\beta}{4(1-\delta-\lambda^R)-c^2}}$ and $\varphi_L(x)=\varphi_0(x/L)$. Then
\[
\begin{cases}
-\Delta_{x}\varphi_L=(1-\delta-\lambda^R-\frac{c^2}{4})\varphi_L\quad\mbox{on }B_L:=\{x\in\R^m: |x|<L\},\\
\varphi_L>0\quad\mbox{in }B_L,\quad \varphi_L=0\quad\mbox{on }\partial B_L, \quad \|\varphi_L\|_{L^\infty(B_L)}=1.
\end{cases}
\]
Define
\[
w_n=
\begin{cases}
e^{-c_n (x\cdot e_n+L)/2}\varphi_L(x)\psi^R(y)\quad\mbox{when }(x,y)\in S_{L,R}=\{(x,y): |x|<L, |y|<R\},\\
0\quad\mbox{elsewhere}.
\end{cases}
\]
It is easy to check that for $(x,y)\in S_{L,R}$, we have
\[
-c_ne_n\cdot \nabla_x w_n-\Delta w_n +\alpha g w_n=\left(\frac{c_n^2-c^2}{4}+1-\delta\right)w_n\le (1-\delta)w_n.
\]

Since $u(1,\cdot,\cdot)$ is continuous and positive in $\R^{m+N}$, there exists $\eta>0$ such that
\[
u(1,x,y)\ge\eta>0\quad\mbox{for all }|x|\le L+c+1, |y|\le R+1.
\]
Then
\begin{equation}\label{eq: v lower}
v_n(-t_n+1,x,y)=u(1,x+c_ne_n,y)\ge\eta \quad\mbox{for all }|x|\le L+1, |y|\le R+1.
\end{equation}
Since $\int_{\R^N} v_n(t,x,y)K(y)\dd y\to 0$ locally uniformly as $n\to\infty$, we define, for $n>J$ (large),
\[
t^*_n=\inf\{t\in [-t_n+1,0]: 0\le\int_{\R^N} v_n(s,x,y)K(y)\dd y\le\delta\mbox{ in }[t,0]\times \{x:|x|\le L+1\}\}.
\]
We may assume $t^*_n<0$. By continuity, we have
\begin{equation}\label{eq:v integral bound}
0\le\int_{\R^N} v_n(t,x,y)K(y)\dd y\le\delta\mbox{ in }[t^*_n,0]\times  \{x:|x|\le L+1\},
\end{equation}
and
\begin{equation}\label{eq:max achieved}
\mbox{if}\quad t^*_n>-t_n+1\quad\mbox{ then }\quad\max_{|x|\le L+1}\int_{\R^N} v_n(t^*_n, x,y)K(y) \dd y=\delta.
\end{equation}

We claim that 
there exists some $\rho>0$ such that
\begin{equation}\label{eq:vn uniform bound}
\min_{|x|\le L, |y|\le R} v_n(t_n^*,\cdot,\cdot)\ge \rho\quad \mbox{ for all } \quad n>J.
\end{equation}
Let us postpone the proof of this claim, and use it to prove \eqref{eq:lower bound speed}. By \eqref{eq:for vn} and \eqref{eq:v integral bound} we have,
\[
\partial_t v_n - \Delta v_n -c_n e_n\cdot \nabla_x v_n + \alpha g v_n\ge\left(1 - \delta\right) v_n\quad\mbox{in }[t^*_n,0]\times S_{L,R}.
\]
By the comparison principle, we have
\[
v_n(t,x,y)\ge \rho w_n(x,y)\quad\mbox{in } [t^*_n,0]\times S_{L,R}.
\]
Since $|y_0|\le R/2$, we have
\[
u(t_n,x_n,y_0)=v_n(0,0,y_0)\ge \rho w_n(0,y_0)=\rho e^{-c_n L/2}\varphi_L(0)\psi^R(y_0)\ge \rho e^{-c L/2}\varphi_L(0)\psi^R(y_0),
\]
which is in contradiction with  $u(t_n,x_n,y_0)\to 0$ as $n\to\infty$.

So it only remains to show \eqref{eq:vn uniform bound}. 

If \eqref{eq:vn uniform bound} fails, after extraction of a subsequence, there exists a sequence $\{(x_n,y_n)\}$ in $\overline S_{L,R}$ such that
\[
v_n(t^*_n,x_n,y_n)\to 0\quad\mbox{and}\quad (x_n,y_n)\to (\bar x,\bar y)\in \overline S_{L,R}.
\]
Define
\[
V_n(t,x,y)=v_n(t+t^*_n,x,y)\quad\mbox{for all }(t,x,y)\in (-t_n-t^*_n,+\infty)\times\R^{m+N},
\]
which satisfies \eqref{eq:for vn} as well. Notice that $-t_n-t^*_n\le -1.$ As before, up to extracting a  subsequence, $V_n$ converges locally uniformly to $V_\infty$ which is a bounded solution of
\[
\partial_t V_\infty - \Delta V_\infty-c_\infty e_\infty\cdot \nabla _x V_\infty + \alpha g V_\infty=\left(1 - \int_{\R^N} V_\infty(t, x,z)K(z) \dd z \right) V_\infty
\]
in $(-1,+\infty)\times\R^{m+N}$. Moreover,
\[
V_\infty(t,x,y)\ge 0\quad\mbox{for all }(t,x,y)\in  (-1,+\infty)\times\R^{m+N}\quad\mbox{and }V_\infty(0,\bar x,\bar y)=0.
\]
By the strong maximum principle, we have $V_\infty\equiv 0$ in $(-1,0]\times\R^{m+N}$, and consequently, by the comparison principle, $V_\infty\equiv 0$ in $[0, +\infty)\times\R^{m+N}$. Hence $V_n$ converges locally uniformly to $0$ in $(-1,+\infty)\times\R^{m+N}$. By Lemma \ref{lem:uniform weighted bound}, $\int_{\R^N} V_n(t, x,y)K(y) \dd y$ also converges to $0$ locally uniformly.  Hence 
$v_n(t^*_n,\cdot,\cdot)\to 0$ and $\int_{\R^N} v_n(t^*_n, \cdot,y)K(y) \dd y \to 0$ locally uniformly. This contradicts \eqref{eq: v lower} and \eqref{eq:max achieved}. The proof is thereby complete.
\end{proof}

\subsection{Asymptotic speed of propagation to $V(y)$}
To complete the proof of Theorem~\ref{t:into-asp}, it remains to prove the following  sharper statement.
\begin{thm}\label{convergence}
 Let $u(t,x,y)$ be as in Theorem \ref{thm:asp}. Assume $0< \alpha < \bar\alpha$. Then, for every $0\le c<c^*$, we have
\[
\lim_{t\to\infty}\sup_{|x|\le ct}|u(t,x,y)-V(y)|=0\quad\mbox{uniformly in }y,
\] 
where $V$ is the unique solution of \eqref{eq:steady} given by \eqref{eq:steady solution}.
\end{thm}

To prove Theorem \ref{convergence}, we will use the same decomposition as in Section \ref{sec:reduction}. We know from Lemma \ref{lem:uniform weighted bound} that $u$ can be written as 
\begin{equation}\label{e:dec-t}
u(t,x,y) = \sum_{i=0}^\infty v_i(t,x) \psi_i(y),
\end{equation}
where $\psi_i$ are those in Lemma \ref{l:spectral}. Then for each $i$, $v_i$ solves
\begin{equation} \label{e:eqvi-t}  
 \partial_t v_i - \Delta v_i + \lambda_i v_i = (1-b(t, x)) v_i,
\end{equation}
where $\lambda_i$ are those in Lemma \ref{l:spectral}, and
\[
b(t,x) := \int_{\R^N} K(z) u(t,x, z) \dd z.
\]

We start with a lemma.
\begin{lemma}
For each $j$ for which $\la_j >1$ the function $v_j(t,x)$ converges exponentially to 0
as $t\to\infty$, uniformly in $x$.
\end{lemma}
\begin{proof}
From the equation \eqref{e:eqvi-t}, the function $v_j$ satisfies 
 \[
 \partial_t v_i - \Delta v_i + \gamma v_i = 0,
 \]
 where $\gamma(t,x) \geq \gamma_0 >0$ for all $t,x$. Lemma~\ref{lem:uniform weighted bound} and the comparison principle yield
\begin{equation}\label{eq:exp decay vj}
|v_j(t,x)| \leq M e^{-\gamma_0 t}
\end{equation}
for some constant $M>0$.
\end{proof}

\begin{lemma}\label{lem: the tail convergence}
Let $J\geq 1$ be an integer such that $\la_{J+1} >1$. Let $z_J$ be defined by
\[
z_J= \sum_{i=J+1}^\infty v_i(t,x) \psi_i(y).
\]
Then, 
$z_J(t,x,y)$ converges to 0 as $t\to\infty$, uniformly in $x$ and locally uniformly in $y$.
\end{lemma}
\begin{proof}
To start with,  we have a bound for all $t>0,x\in\R^m$,
\[
\sum_{i=0}^\infty v_j^2(t,x)=\int_{\R^N}u^2(t,x,y)\dd y \leq C.
\]
for some constant $C>0$. Therefore,
\[
\|z_J(t,x,\cdot)\|_{L^2(\R^N)}\le C.
\]
We claim that
\[
z_J(t,x,\cdot)\rightharpoonup 0\quad\mbox{weakly in }L^2(\R^N) \quad\mbox{as }t\to\infty.
\]
Indeed, for every $\varphi\in L^2(\R^N),$ for every $\eps>0$,  there exist an integer $\ell>0$ and $\mu_i\in\R$ such that
\[
\|\varphi-\sum_{i=0}^\ell \mu_i\psi_i\|_{L^2(\R^N)}\le\eps.
\]
Therefore,
\[
\begin{split}
&|\int_{\R^n}z_J(t,x,y)\varphi(y)\dd y|\\
&\le \sum_{i=0}^\ell  |\int_{\R^n}z_J(t,x,y)\mu_i\psi_i\dd y|+ |\int_{\R^n}z_J(t,x,y)(\varphi-\sum_{i=0}^\ell\mu_i\psi_i)\dd y|\\
&\le \sum_{i=J+1}^\ell |\mu_i v_i(t,x)|+C\eps\\
&\le 2C\eps \quad \mbox{for all}\quad  t\ge T,
\end{split}
\]
 where $T$ is sufficiently large but independent of $x$, and we used \eqref{eq:exp decay vj} in the last inequality. This proves the claim.

Let $R>0$. Since $z_J=u-\sum_{i=0}^Jv_i\psi_i$ is  Lipschitz continuous uniformly in $(t,x)$ and locally uniformly in $y$ for $t>1$, for every $\eps>0$, there exists $r>0$ such that for all $t>1, x\in\R^m, |y|\le R$, the oscillation of $z_J$ in the ball $B_r(y)$ satisfies the bound:
\[
\osc_{B_r(y)}z_J(t,x,\cdot)\le\eps.
\]
By the weak convergence, there exists $T=T(y)>0$ such that for all $x\in\R^m$ and all $t>T$,
\[
|\frac{1}{|B_r(y)|}\int_{B_r(y)}z_J(t,x,z)\dd z|\le\eps.
\]
Consequently, we have for $|y|\le R$,
\[
|z_J(t,x,z)|\le 2\eps\quad \mbox{for all}\;\,  z\in B_{r}(y),\;\, \mbox{all}\;\,  t>T=T(y), \;\, \mbox{and all}\;\,  x\in \R^m.
\]
Therefore, we can conclude that
\[
z_J(t,x,y)\to 0\quad \mbox{as }t\to\infty\quad\mbox{uniformly in }x \mbox{ and locally uniformly in }y.
\]
\end{proof}

Therefore, to prove Theorem \ref{convergence}, we only have to deal with the finite sum
$u-z_J$. We are going to prove that the finite number of functions $v_1(t,x), \ldots, v_{J}(t,x) $ converge to 0. 

Let
$w_i: = v_i/v_0$. 
Using the equation \eqref{e:eqvi-t} we derive an equation for $w_i$:
 \begin{equation}\label{eq:wi-tx}
  \partial_t w_i - \lap w_i -\frac{2\nabla v_0}{v_0} \cdot \nabla w_i + (\lambda_i-\lambda_0)w_i = 0.
 \end{equation}
 
 Owing to Theorem \ref{thm:asp}, we know that $\liminf_{t\to+\infty} v_0(t,x) >0$ locally uniformly in $x$. Actually, we have a stronger information on the limit. Let $0<\gamma < c^*$. For any  $A>0$, there exists $\rho>0$ such that:
 \begin{equation} \label{convunif-u}
 u(t,x, y) \geq \rho, \quad \text{for all} \quad t\geq 1, |x| \leq \gamma t, |y|\leq A.
 \end{equation}
 Together with \eqref{eq:assumption-2}, this implies the existence of $\delta=\delta(\gamma)>0$ such that
  \begin{equation} \label{convunif}
 v_0(t,x) \geq \delta, \quad \text{for all} \quad t\geq 1, |x| \leq \gamma t.
 \end{equation}
 
 \begin{prop}\label{prop: finitely many convergence}
 Let $0\le c<c^*$. For each $j=1, \ldots, J$, we have
 \[
 \lim_{t\to+\infty} \sup_{|x|\le ct}|v_j(t,x)|= 0.
 \]
 \end{prop}
 \begin{proof}
  We argue by contradiction. Suppose there are $0\le c<c^*$, $\eta>0$, and a sequence $\{(t_n,x_n)\}$ with
\[
 \begin{cases}
  |x_n|\le c t_n\quad\mbox{for all }n\in\mathbb N,\\
   t_n\to\infty\quad\mbox{and } |v_j (t_n,x_n)|\ge\eta\quad\mbox{as }n\to\infty.
 \end{cases}
\]
 Choose $\gamma$ such that $c<\gamma<c^*$. 
We translate the functions $v_0$, $w_j$ in time and space to define
\[
V_n(t,x) := v_0(t_n+t, x+x_n), \quad \text{and} \quad W_n(t, x) := w_j(t_n+t, x+x_n)
\]
for $t\geq -t_n$ and $x\in\R^m$. 
We also have the equation:
\begin{equation}\label{eq:Wk-tx}
  \partial_t W_n - \lap W_n -\frac{2\nabla V_n}{V_n} \cdot \nabla W_n + (\lambda_j-\lambda_0)W_n = 0.
 \end{equation}
From \eqref{convunif}, we observe that $V_n$ is bounded from below by $\delta>0$ on the larger and larger set $\Omega_n:=\{(t,x): |x|\le(\gamma-c)t_n +\gamma t\}$ as $n\to\infty$. Therefore, $W_n$ and the coefficient of the gradient term $\frac{2\nabla V_n}{V_n} $ are bounded on  $\Omega_n$ as $n\to\infty$. By parabolic estimates, up to striking out a subsequence, we obtain the convergence of  
$V_n$ and $W_n$ locally uniformly  to $V_\infty$ and $W$. The limit functions satisfy an equation defined for all $t\in\R$ and $x\in\R^m$:
\begin{equation}\label{eq:Wlim}
  \partial_t W- \lap W -\frac{2\nabla V_\infty}{V_\infty} \cdot \nabla W + (\lambda_j-\lambda_0)W = 0,\quad t\in\R, x\in\R^m.
 \end{equation}
Furthermore, we know that
\[
|W(0, 0) |>0 , \quad \text{and} \quad V_\infty(t, x) \geq \delta \quad \forall \ t\in\R, x\in\R^m.
\]
Therefore, the equation \eqref{eq:Wlim} has bounded coefficients. Moreover, $W$ is bounded since $V_\infty$ is bounded from below,  and $W$ is a time-global solution (i.e. defined for all $t$). Denote $M:=\sup_{\R\times\R^m} |W(t,x)|$. The function
\[
e^{(\la_0 - \la_j) t}
\] 
is a solution of  \eqref{eq:Wlim}. From the comparison principle applied to \eqref{eq:Wlim} we get:
\[
|W(t,x)| \leq M e^{(\la_0 - \la_j) (t-\tau)} \quad \text{for all} \quad t\geq \tau.
\]
Letting $\tau\to -\infty$, we get $W(t, x) = 0$ for all $t$ and all $x$. This is in contradiction with the value of $W$ at  $(0, 0)$. The proof of the proposition is thereby complete.
\end{proof}

\begin{proof}[Proof of Theorem \ref{convergence}] 
We are going to first derive the following limit:
\begin{equation}\label{eq:v0 asymptotic}
\sup_{|x|\le ct} |v_0(t,x)- \mu|\to 0 \quad\mbox{as } t\to\infty,
\end{equation}
where $\mu$ is given in \eqref{eq:-infty}. Letting
\[
\tilde b(t,x)=-\int_{\R^N}K(z) [ u(t,x,z)-v_0(t,x)\psi_0(z) ]\dd z,
\]
we have the following equation:
\[
\partial_t v_0-\Delta v_0=(\tilde b(t,x)+1-\lambda_0-\mu^{-1}(1-\lambda_0)v_0)v_0,
\]

The proof of \eqref{eq:v0 asymptotic} is similar to the proof of Proposition \ref{prop: finitely many convergence}. Suppose there are $0\le c<c^*$, $\eta>0$, and a sequence $\{(t_k,x_k)\}$ such that
\[
 \begin{cases}
  |x_k|\le c t_k\quad\mbox{for all }n\in\mathbb N,\\
   t_k\to\infty\quad\mbox{and } |v_0 (t_k,x_k)-\mu|\ge\eta\quad\mbox{as }n\to\infty.
 \end{cases}
\]
Choose $\gamma$ such that $c<\gamma<c^*$.  From Lemma \ref{lem: the tail convergence}, Proposition \ref{prop: finitely many convergence}, Lemma \ref{lem:uniform weighted bound} and the dominated convergence theorem, we infer that
\[
\lim_{t\to\infty} \sup_{|x|\le \gamma t}\tilde b(t,x)=0.
\]
Define
\[
V_k(t,x)=v_0(t_k+t, x+x_k),\quad B_k(t,x)=\tilde b(t_k+t, x+x_k),
\]
which satisfies
\[
\partial_t V_k-\Delta V_k=(B_k(t,x)+1-\lambda_0-\mu^{-1}(1-\lambda_0)V_k)V_k.
\]
From \eqref{convunif}, we observe that $V_k\ge \delta$ on the larger and larger set $\{(t,x): |x|\le(\gamma-c)t_k+\gamma t\}$ as $k\to\infty$. Moreover, $B_k(t,x)\to 0$ locally uniformly as $k\to\infty$.
From the upper bound of $v_0$ we can strike out a subsequence such that $V_k$ converges locally uniformly to a bounded function $V_\infty$, which satisfies $V_\infty\ge\delta$ in $\R\times\R^m$, and 
\begin{equation}\label{eq:convergence v0}
\partial_t V_\infty-\Delta V_\infty=(1-\lambda_0-\mu^{-1}(1-\lambda_0)V_\infty)V_\infty\quad\mbox{for all }(t,x)\in \R\times\R^m.
\end{equation}
We claim that 
\[
V_{\infty}\equiv \mu\quad\mbox{in } \R\times\R^m.
\]
The proof of the claim is as follows. Let $M=\sup_{\R\times\R^m} V_\infty$ and $(t_n,x_n)\in \R\times\R^m$ be such that $V_\infty(t_n,x_n)\to M$ as $n\to\infty$. Let $W_n(t,x)=V_\infty(t+t_n,x+x_n)$.  Then subject to a subsequence, $W_k$ converges locally uniformly to $W_\infty$, which also satisfies \eqref{eq:convergence v0}. Moreover, $W_\infty(0,0)=M=\sup_{\R\times\R^m}W_\infty$. By evaluating at $(0,0)$, we have $M\le \mu$, i.e., $\sup_{\R\times\R^m} V_\infty\le \mu$. Similarly, one can show that $\mu\le \inf_{\R\times\R^m} V_\infty.$
Therefore, $V_\infty\equiv \mu$. 

Hence, $v_0(t_k,x_k)=V_k(0,0)\to\mu$, which is a contradiction. This proves \eqref{eq:v0 asymptotic}.

Once we have the limit \eqref{eq:v0 asymptotic}, Lemma \ref{lem: the tail convergence} and Proposition \ref{prop: finitely many convergence} yield 
\[
\lim_{t\to\infty}\sup_{|x|\le ct}|u(t,x,y)-V(y)|=0\quad\mbox{locally uniformly in}\quad y.
\]
Since both $u$ and $V$ are uniformly exponentially decaying as $|y|\to\infty$, we can conclude that the above convergence is uniform in $y$.
\end{proof}

The proof of Theorem \ref{t:into-asp} is thereby complete. 
\hfill \qed

\appendix

\section{Local maximum principle}
In this appendix, we provide a short proof of the local maximum principle for heat equations, which was used in the proof of Lemma \ref{lem:uniform bound}. The following statement and its proof are well-known, and we include them here for the purpose of completeness. 

\begin{lemma}\label{lem:local max principle}
Let $u\in C^{2,1}_{x,t}(Q_1)$ be a nonnegative solution of
\[
u_t-\Delta u +c(t,x) u\le 0\quad\mbox{in }Q_1:=(0,1]\times B_1,
\]
where $c(t,x)$ is a bounded function in $Q_1$. Then there exists a positive constant $C$ depending only on $n$ and $\|c^-\|_{L^\infty(Q_1)}$ such that
\[
u(t,x)\le C\int_{Q_1} u\quad\mbox{for all }(t,x)\in [15/16,1]\times B_{1/4}.
\]
\end{lemma}
\begin{proof}
Let $c_0=\|c^-\|_{L^\infty(Q_1)}$ and $\tilde u=e^{-c_0t}u$. Then
\[
\tilde u_t-\Delta \tilde u\le e^{-c_0t}(-c_0 -c(t,x))u\le 0\quad\mbox{in }Q_1.
\]
Let $\eta(t,x)$ be a smooth nonnegative cut-off function such that $\eta\equiv 1$ in $[3/4,1]\times B_{1/2}$, $\eta\equiv 0$ in $([1/2,1]\times B_{3/4})^c$ and $0\le\eta\le 1$. Let $v=\eta \tilde u$. Then it satisfies that
\[
v_t-\Delta v\le (\eta_t-\Delta\eta)\tilde u -2\nabla\eta\nabla \tilde u=:f(t,x).
\]
Let
\[
G(t,x)=\frac{1}{(4\pi t)^{n/2}}e^{-\frac{|x|^2}{4t}}
\]
be the heat kernel.  Therefore, we have
\[
v(t,x)\le \int_0^t\int_{\R^n}G(t-s,x-y)f(s,y)\dd y\dd s.
\]
Then the conclusion follows from integration by parts and the observation that $f\equiv 0$ in $[3/4,1]\times B_{1/2} \cup ([1/2,1]\times B_{3/4})^c$.
\end{proof}

\bibliographystyle{plain}

\index{Bibliography@\emph{Bibliography}}%

\bigskip

\noindent H. Berestycki

\noindent \'Ecole des hautes ́\'etudes en sciences sociales\\
CAMS, 190-198, avenue de France, 75244 Paris cedex 13, France\\[1mm]
Email: \textsf{hb@ehess.fr}

\medskip

\noindent T. Jin

\noindent Department of Mathematics, The University of Chicago\\
5734 S. University Avenue, Chicago, IL 60637, USA\\[1mm]
Email: \textsf{tj@math.uchicago.edu}

\medskip

\noindent L. Silvestre

\noindent Department of Mathematics, The University of Chicago\\
5734 S. University Avenue, Chicago, IL 60637, USA\\[1mm]
Email: \textsf{luis@math.uchicago.edu}

\end{document}